\documentclass[12pt]{amsart}
\usepackage[utf8]{inputenc}

\usepackage[colorlinks, citecolor=blue, linkcolor=red]{hyperref}
\usepackage{amsthm}
\usepackage{amssymb}
\usepackage{verbatim}
\usepackage{amsmath}
\usepackage{amscd}
\usepackage{latexsym}
\usepackage{xcolor}
\usepackage{mathrsfs}

\newcommand{\R}{{\mathbb R}}
\newcommand{\disp}{\displaystyle}
\newcommand{\U}{\mathcal U}
\newcommand{\dive}{\text{\rm div }}
\newtheorem{thm}{\textbf{Theorem}}[section]
\newtheorem{lem}[thm]{\textbf{Lemma}}
\newtheorem{pro}[thm]{\textbf{Proposition}}
\newtheorem{rem}[thm]{\textbf{Remark}}
\newtheorem{cor}[thm]{\textbf{Corollary}}
\newtheorem{defn}[thm]{\textbf{Definition}}
\theoremstyle{remark}

\title[Matrix weights and regularity results]{Matrix weights and regularity for degenerate elliptic equations}
\author[Di Fazio, Fanciullo, Monticelli, Rodney, Zamboni]{Giuseppe Di Fazio, Maria Stella Fanciullo, Dario Daniele Monticelli, Scott Rodney, Pietro Zamboni}
\date{}
\begin{document}

\begin{abstract}
We prove local boundedness, Harnack inequality and local regularity for weak solutions of quasilinear degenerate elliptic equations in divergence form. Degeneracy is via a non negative, symmetric, measurable matrix-valued function $Q(x)$ and two suitable non negative weight functions. We setup an axiomatic approach in terms of suitable geometric conditions and local Sobolev -- Poincar\'{e} inequalities. Data integrability is close to $L^1$ and it is exploited in terms of suitable version of Stummel-Kato class that in some cases is also necessary to the regularity.
\end{abstract}

\maketitle

\section{Introduction}

The aim of this note is to study the local behaviour of  solutions to quasilinear degenerate elliptic equations of the following kind
\begin{align}\label{Qppoisson1}-\textrm{div}\left(h(x)\left|\sqrt{Q(x)}\nabla u(x)\right|^{p-2}Q(x)\nabla u(x)\right) = m(x)|f(x)|^{p-2}f(x)
\end{align}
that admits two different types of degeneracy.

The first one is given by a non negative definite, symmetric, measurable matrix-valued function $Q(x)$ associated with the principal part of the differential operator. The other one is expressed in terms of two non negative weight functions $h$ and $m$ that can vanish or blow up.

We  assume that $f$ lies in a natural 2-weight generalization of the Stummel-Kato class (see \cite{as} and \cite{CFG}) and  prove local boundedness and Harnack inequality for weak solutions of equation \eqref{Qppoisson1} (see Theorems \ref{localboundedness} and \ref{lastbutone}).
We accomplish this using Moser's iteration technique (see \cite{s}) and a Fefferman-Phong inequality adapted to our setting.
As a consequence, we obtain local H\"older continuity of weak solutions (see Theorem \ref{lastone} and Remark \ref{lastrem}).

Each of our results is in the spirit of \cite{DMSR1} and \cite{MRW1} where the authors investigate quasilinear degenerate elliptic equations of the type
\begin{equation}\label{equa_introd}
\dive A(x,u,\nabla u)+B(x,u,\nabla u)=0\,
\end{equation}
for functions $A:\mathbb{R}\times\mathbb{R}\times\mathbb{R}^n \rightarrow \mathbb{R}^n$ and $B:\mathbb{R}\times\mathbb{R}\times\mathbb{R}^n \rightarrow \mathbb{R}$ that satisfy Trudinger - Serrin type structural conditions.  In the case $h=m=1$, our equation \eqref{Qppoisson1} is also of this form.  However, in \cite{DMSR1} and \cite{MRW1}, the data function $f$ is required to satisfy minimal $L^p$ or Morrey type restrictions stated in terms of the gain in a Sobolev-type inequality.  As such, our work generalizes \cite{DMSR1} and \cite{MRW1} for the special case of equation \eqref{Qppoisson1} in two ways.  First, we use an approach that allows the incorporation of the weights $h,m$, and second, we allow a new non-$L^p$ restriction on the data function $f$. \\

Concerning the study of the regularity properties of solutions of quasilinear elliptic equations of the kind \eqref{equa_introd}, the first paper where non - $L^p$ conditions appear is \cite{RAZ}, where a simplified form of \eqref{equa_introd} was studied with right hand side being a measure in a Morrey space.  Local H\"older continuity of weak solutions is proved using De Giorgi's methods, as adapted by Ladyzhenskaya and Ural'tseva in \cite{lu} to quasilinear equations.  Improvements of \cite{RAZ} can be found in
\cite{RA} and \cite{Lieberman} where $L^p$ assumptions are weakened or replaced by Morrey space restrictions. Similar investigations were made in \cite{Z4} with Stummel-Kato type requirements on the lower order coefficients generalizing \cite{as} and \cite{CFG} to quasilinear elliptic equations.

Aside from minimal data restrictions, many authors have also investigated the regularity of weak solutions of \eqref{Qppoisson1} where degeneracy is encoded by the weights $h,m$ and the matrix function $Q$.  For example, when $Q(x)$ has non negative minimum eigenvalue given by an appropriate $A_\infty$ weight, local regularity of solutions was studied by Fabes, Kenig, and Serapioni in \cite{FKS} with further progress in \cite{DKV} among others.  In each of these works, continuity results are obtained using a Sobolev inequality adapted to $Q$ and an iterative argument of Moser type.   Even when the minimum eigenvalue is not $A_\infty$, some progress has been made in \cite{SRKS} where they allow the matrix $Q$ to vanish to inifinite order at an interior point of a domain.  Similarily,  \cite{dz} and \cite{dfz3} obtain similar results for data in appropriate weighted Stummel-Kato type classes.

Many regularity results for weak solutions of elliptic equations have also been generalized to the
Carnot Caratheodory (CC) spaces where a metric is generated using sub-unit curves associated
to a system of non-commuting vector fields, and these vector fields define the matrix $Q$. In this direction we quote \cite{CDG},
\cite{D}, \cite{DGN}, \cite{dz1} and \cite{dz2} where the above
mentioned results are obtained in the subelliptic setting through Morrey and Stummel-Kato
type assumptions on lower order terms.

The first attempt to consider together the several types of degeneracy described above has been exploited by
Franchi, Gutierrez and Wheeden in \cite{fgw1} and \cite{fgw2}, where they proved Harnack inequality for non negative weak solutions of equation
$$
\dive (w(z)\nabla_\lambda u(z))=0
$$
where $w$ is a suitable power of a strong $A_\infty$ weight, the $\lambda$ gradient is
$\nabla_\lambda u (z)=(\nabla_xu(z)$, $\lambda(x)\nabla_yu(z))$ and $z=(x,y)$.
We remark that equations studied  in \cite{fgw1} and \cite{fgw2} do not contain lower order terms.  We also note that \cite{28} and \cite{35} extend the  results contained in \cite{fgw1} and \cite{fgw2} to more general equations with lower order terms in suitable Stummel-Kato classes.\\

The paper is organized as follows: in Section \ref{sec2} we collect the axiomatic assumptions under which we prove our results concerning weak solutions of equation \eqref{Qppoisson1} and we provide some preliminary results; in Section \ref{sec3} we describe our functional setting; in Section \ref{sec4} we introduce the Stummel-Kato class adapted to our setting and we prove a Fefferman-Phong inequality; lastly, in Section \ref{sec5} we state and prove our local regularity results and the Harnack inequality for weak solutions of equation \eqref{Qppoisson1}.

\section{Preliminaries}\label{sec2}

Fix a bounded domain $\Omega$ of $\mathbb{R}^n$ and assume it is endowed with a metric $\rho$ that generates a topology equivalent to the Euclidean one.
We will denote respectively by $$B_r(x)=B(x,r)=\{y\in\Omega\,|\,\rho(x,y)<r\}$$ and $D_r(x)=D(x,r)$ respectively the $\rho$-metric ball and the Euclidean ball centred at $x\in\Omega$ with radius $r>0$.  By our assumptions, every metric ball $B$ is an open set; moreover, for every small enough radius $r$, the Euclidean closure of the metric ball $B(x,r)$ is contained in $\Omega$.

Indeed, we have

\begin{lem}
   Let $E,E_0$ be bounded open domains such that $\overline{E}\subset E_0$, $\overline E_0\subset\Omega$. Then there exists $\delta>0$ such that for every $x\in E$, $0<r<\delta$ one has
   \begin{equation}\label{1}
       \overline{B(x,r)}\subseteq\{y\in\Omega\,|\,\rho(x,y)\leq r\}\subset E_0
       \end{equation}

\end{lem}

\begin{proof}
  Arguing by contradiction, if the result does not hold one can find two sequences $x_n$, $y_n$ such that for every $n\in\mathbb{N}$ one has $x_n\in E$, $y_n\in\partial E_0$ and $\rho(x_n,y_n)<\frac{1}{n}$. Since $\overline{E}$ is compact, up to a subsequence we have that $x_n$ converges to some $\bar{x}\in\overline{E}$, in the Euclidean sense. Since the topologies are equivalent we see that $\rho(\bar{x},x_n)$ tends to $0$, and thus by the triangle inequality we deduce that also $\rho(\bar{x},y_n)$ tends to $0$. Hence, again by the equivalence of the topologies, we have that $y_n$ converges to $\bar{x}$, so that $\bar{x}\in\partial E_0$. Then $\overline{E}\cap\partial E_0\neq\emptyset$, a contradiction.
\end{proof}

We will write $r=r(B)$ to denote the radius of a given ball $B$.  For simplicity of presentation, given any positive constant $c$, we denote by $cB$ the ball concentric to $B$ with radius $cr(B)$. Given any two balls $B_1=B(x,r_1)$, $B_2=B(y,r_2)$ with non-empty intersection and radii $r_1\geq r_2$, we remind the reader that one has $B_2\subset 3B_1$, the standard swallowing result for metric balls.

With a view to studying regularity for degenerate elliptic operators of second order with rough coefficients, let $\mu$ be a given measure and $Q:\Omega\rightarrow S_n$ be a $\mu$-measurable matrix function taking values in $S_n$, the collection of all non negative definite $n\times n$ symmetric matrices.  We will always assume that the $p^{th}$ power of the pointwise operator norm of $\sqrt{Q}$ is a weight - that is, $v=\|Q(x)\|_{op}^{p/2}\in L^1_{loc}(\Omega;\mu)$, where $$\|Q(x)\|_{op}:=\sup_{\xi\in\mathbf{R}^n,|\xi|=1}\langle\xi,Q(x)\xi\rangle.$$  As an application, we consider principal second order equations of the form
\begin{align}\label{Qppoisson}-\textrm{div}\left(h(x)\left|\sqrt{Q(x)}\nabla u(x)\right|^{p-2}Q(x)\nabla u(x)\right) = m(x)|f(x)|^{p-2}f(x)
\end{align}
for $x\in\Omega$ and $f$ in the Stummel-Kato class of functions defined below.

We will always assume that the non negative weight functions $h$ and $m$ are in $L^1_{loc}(\R^n)$.

We will be working in an axiomatic setting where we will assume the validity of geometric and analytic assumptions that we list below.

 \begin{itemize}
\item[\bf{(H1)}]  \textit{Segment property.} For every $x,y\in\Omega$ there exists a continuous curve $\gamma:I\subset\mathbb{R}\rightarrow\Omega$ connecting $x,y$ such that $\rho(\gamma(s),\gamma(t))=|s-t|$ for every $s,t\in I$.
\item[\bf{(H2)}]  The $\sigma$-finite measures $\mu=h\,dx$, $\nu=m\,dx$ are both doubling for balls $B$ with $2B \Subset\Omega$.  Hence, there are constants $A_\nu,A_\mu>0$ so that for any such ball
\begin{eqnarray}\label{nudoub}
\nu(2B) \leq A_\nu \nu(B)\qquad \mu(2B)\leq A_\mu\mu(B).
\end{eqnarray}
We also assume that $\mu$ is reverse-doubling of order 1, i.e. there exists a constant $C>0$ such that for every pair of balls $\tilde{B}\subset B\Subset\Omega$ one has $$\mu(B)\geq C \left(\frac{r(B)}{r(\tilde{B})}\right)\mu(\tilde{B}).$$
\item[\bf{(H3)}] We require that $d\mu\ll d\nu$, where $d\mu=h\,dx$ and $d\nu=m\,dx$. Moreover, we also assume that $\nu(E)>0$ for every Lebesgue measurable set $E\subset\Omega$ having positive Lebesgue measure. This is equivalent to assume that $m>0$ Lebesgue-a.e. Finally we require that there exists $C>0$ such that for every compact subset $E\subset\Omega$ we have
\begin{equation}\label{ipotesi_misure}
\int_E v~d\mu\leq C\int_E ~d\nu
\end{equation}
where we recall that $v=\|Q\|^{p/2}_{op}$. Note that this is equivalent to $vh\leq C m$ Lebesgue-a.e.
\item[\bf{(H4)}] \textit{Local Poincar\'{e} inequality.} There exists a constant $C_0>0$ so that given any ball $B\Subset\Omega$ one has
\begin{equation}\label{P}
\frac{1}{\nu (B)}\int_B|u-u_B|~d\nu\le C_0 \frac{r(B)}{\mu(B)}\int_B|\sqrt Q\nabla u|~d\mu\end{equation}
for every $u\in Lip_{loc}(\Omega)$, the collection of all locally Lipschitz functions defined on $\Omega$, where $u_B=\frac{1}{\nu(B)}\int_B u \,d\nu$.
\item[\bf{(H5)}] \textit{Local Sobolev inequality.} For any $p\geq 1$ there exists constants $C_1>0$, $k>1$ so that given any ball $B\Subset\Omega$ one has
\begin{equation}\label{S}
\left(\frac{1}{\nu (B)}\int_B|u|^{kp}~d\nu\right)^\frac{1}{kp}\le C_1 r(B)\left(\frac{1}{\mu(B)}\int_B|\sqrt Q\nabla u|^p~d\mu\right)^\frac{1}{p}\end{equation}
for every $u\in Lip_0(B)$, the collection of all Lipschitz functions compactly supported in $B$.
\item[\bf{(H6)}]  \textit{Accumulating sequences of Lipschitz cut-off functions.}

Given any metric ball $B=B(x,r)$ such that $B\Subset\Omega$ there exists a sequence of functions $\varphi_j\in Lip_0(B)$ so that for every $j\in\mathbb{N}$
    \begin{enumerate}
    \item $0\leq \varphi_j(x)\leq 1$,
    \item $\varphi_j(x)=1$ for $x\in \tau B$,
    \item $\operatorname{supp}\varphi_{j+1}\subset\{y\in B\,:\, \varphi_j(y)=1\}$,
    \item $\displaystyle\left\|\frac{|\sqrt{Q}\nabla\varphi_j|}{v^\frac{1}{p}}\right\|_\infty\leq \disp\frac{\tilde{C}T^j}{r^\gamma}$.
    \end{enumerate}
where $\gamma,\tilde{C}>0,T\geq1$, $0<\tau<1$ are positive constants.
\end{itemize}

\begin{rem} Since $\nu$, $\mu$ are doubling measures on the collection of (open) metric balls by (H2), it is easy to see that no open set can have measure zero.

Concerning condition (H6)-(4), we note that, by the definition of $v$, one always has
$$\left\|\frac{|\sqrt{Q}\nabla\varphi|}{v^\frac{1}{p}}\right\|_\infty\leq\|\nabla\varphi\|_\infty$$ for every $\varphi\in Lip (\Omega)$.
\end{rem}

Note that, since $d\nu$ is doubling for balls $B$ with closure in $\Omega$, we also have that $(\Omega,\rho)$ is  geometrically doubling for $\varrho$ - balls. That is, there exists a constant $M\geq1$ such that if $B$ is a ball satisfying $2B\Subset\Omega$, then it may contain at most $M$ centres of disjoint metric balls of radius $r(B)/2$. The connection between doubling properties for measures and geometric doubling for the collection of balls $\{B\}$ whose doubles have closure in $\Omega$ can be further explored in \cite{HyK,HyM}.

This property is useful on a scale of radii that we present in the following lemma.

\begin{lem}\label{geolem}
If $(\Omega,\rho)$ is a geometrically doubling metric space and the ball $B=B(x,r)$ satisfies $2B\Subset\Omega$, then it may contain at most $Ms^{-Q}$ centers of disjoint balls of radius $sr$, with $0<s<1$, where $Q=\log_2(M)$ and $M$ is as above.
\end{lem}

The following covering lemma will play an important role in the proof of Corollary \ref{sobnogain}. We note that this is a simplified version of \cite[Lemma 3.12]{CRW} but we will prove it for the reader's convenience.

\begin{lem}\label{covering} Fix any $c^*\geq1$. Then given any compact subset $E$ of  $\Omega$ for every $r>0$ small enough there  exists a finite collection of metric balls $\{B_j\}_{j=1}^N=\{B(z_j,r)\}_{j=1}^N$ centred in $E$, with $N=N(r)$, with closure contained in $\Omega$ and such that
\begin{enumerate}
\item $E\subset \bigcup_j B_j$,
\item $\left\{\frac{1}{3}B_j\right\}_j$ is a pairwise disjoint collection,
\item $\sum_{j}\chi_{c^*B_j}(x) \leq (18c^*)^Q$ for every $x\in\Omega$, with $Q$ as in Lemma \ref{geolem} (uniformly finite overlap).
\end{enumerate}
\end{lem}
\begin{proof}
Fix a compact $E\subset \Omega$, and let $r>0$ be small enough so that $B(x,6c^*r)\Subset\Omega$ for every $x\in E$. Since metric balls are open sets, cover $E$ with a finite collection of balls $\{\tilde{B}_j\}_{j=1}^{m} = \{\tilde{B}(z_j,\frac{r}{3})\}_{j=1}^m$.  We extract a pairwise disjoint subcollection $\{\hat{B}_k\}_{k=1}^{N}$ of these via the following process.  Set $\hat{B}_1 = \tilde{B}_1$.  From $\{\tilde{B}_2,...,\tilde{B}_m\}$, select $j$ so that $\hat{B}_1\cap \tilde{B}_j=\emptyset$ and set $\hat{B}_2 = \tilde{B}_j$.  If there is no such $j$, stop the process.  Continue as directed until no ball can be found which has empty intersection with each of the balls previously found.  From this new collection, $\{\hat{B}_k\}_{k=1}^{N}\subset\{\tilde{B}_j\}_{j=1}^{m}$, we set $B_k = 3\hat{B}_k\supseteq \hat{B}_k$ and now claim that $\{B_k\}_{k=1}^{N}$ satisfies items $1-3$.  Item 2 is clear since the balls $\frac{1}{3}{B}_k = \hat{B}_k$ form a pairwise disjoint family by construction.  To see item $1$, let $x\in E$ and choose $1\leq j \leq m$ so that $x\in \tilde{B}_j$.  If $\tilde{B}_j = \hat{B}_k$ for some $1\leq k\leq N$ we are finished.  If not, then there  exists a $1\leq k\leq N$ so that $\tilde{B}_j\bigcap \hat{B}_k\neq \emptyset$.  Then we see that $x\in 3 \hat{B}_k = B_k$.  We conclude in either situation that $x\in \bigcup_j B_j$ proving item $1$.

To address item $3$, suppose $\{k_s\}_{s=1}^\Gamma$ is a sub collection of $\{1,..,N\}$ with the property that $$\bigcap_{s=1}^{\Gamma} c^*B_{k_s} \neq\emptyset.$$  Then $c^*B_{k_s} \subset 3c^*B_{k_1}$ for each $s$.  Since $6c^*B_{k_1}\Subset\Omega$ and $\{\frac{1}{3}{B}_{k_s}\}_s$ is pairwise disjoint, Lemma \ref{geolem} shows there are at most $\left(18c^*\right)^Q$ such balls showing that $\Gamma \leq \tilde{M} =\left(18c^*\right)^Q$.   This proves item $3$ and completes the proof of this lemma.
\end{proof}

Since $d\nu$ is doubling, we also have the following result.
\begin{lem}\label{2}
  Let $B=B(x_0,r)$ be such that $B(x_0,\frac{11}{2}r)\Subset\Omega$. Then
  \begin{equation}\label{4bis}
  1-\frac{\nu(B)}{\nu(2B)}\geq\frac{1}{1+A_\nu^{3}}>0.
  \end{equation}
\end{lem}
\begin{proof}
  Let $x\in\Omega$ be such that $\rho(x,x_0)=\frac{3}{2}r$. Let $\tilde{B}=B(x,\frac{r}{2})$, then $\tilde{B}\subset 2B\setminus B$ and thus
\begin{equation}\label{3bis}
\nu(2B)\geq\nu(B)+\nu(\tilde{B}).
\end{equation}
Now note that $B\subset8\tilde{B}\subset B(x_0,\frac{11}{2}r)\Subset\Omega$, hence by doubling $$\nu(B)\leq\nu(2^{3}\tilde{B})\leq A_\nu^{3}\nu(\tilde{B}).$$ From \eqref{3bis} we deduce that $$\nu(2B)\geq\left(1+\frac{1}{A_\nu^{3}}\right)\nu(B),$$ whence \eqref{4bis} follows.
\end{proof}

\section{Functional setting}\label{sec3}

\begin{defn}
Let $p>1$. We denote by $QW^{1,p}_{\nu,\mu}(\Omega)=QW^{1,p}(\Omega)$ and $QW^{1,p}_{0,\nu,\mu}(\Omega)=QW^{1,p}_0(\Omega)$ the respective completions of $Lip_{loc}(\Omega)$ and $Lip_0(\Omega)$ subject to the norm
\begin{align}\label{sobonorm}
\|u\|_{QW^{1,p}(\Omega)}&=\left (\int_{\Omega}|u|^p~d\nu\right)^{1/p}+\left(\int_\Omega\langle\nabla u, Q\nabla u\rangle^{p/2}~d\mu\right)^{\frac{1}{p}}\\
\nonumber&=\left (\int_{\Omega}|u|^p~d\nu\right)^{1/p}+\left(\int_\Omega |\sqrt Q \nabla u|^p~d\mu\right)^{\frac{1}{p}}.
\end{align}
\end{defn}

While $QW^{1,p}(\Omega)$ $\left(QW^{1,p}_0(\Omega)\right)$ is defined as a collecton of equivalence classes of sequences of $Lip_{loc}(\Omega)$ $\left(Lip_0(\Omega)\right)$ functions Cauchy with respect to the norm \eqref{sobonorm}, it is well known (see \cite{DCUSR,DCUSRER,DCUSRER1,DMSR,DMSR1,MRW1,MRW2}) that it is isometrically isomorphic to a collection of pairs $(u,\vec{g}) \in L^p_\nu(\Omega)\times \mathcal{L}^p_\mu(Q;\Omega)$ endowed with the norm
\begin{align}\label{pairnorm}
\|(u,\vec{g})\|_{QW^{1,p}(\Omega)}=\left (\int_{\Omega}|u|^p~d\nu\right)^{1/p}+\left(\int_\Omega|\sqrt{Q}~ \vec{g}|^p~d\mu\right)^{\frac{1}{p}},
\end{align}
where we denote by $\mathcal{L}^p_\mu(Q;\Omega)$ the space of measurable vector valued functions $\vec{g}$ such that $$\int_\Omega|\sqrt{Q}~ \vec{g}|^p~d\mu<\infty.$$

While the generalized derivative $\vec{g}$ of $u$ may not depend on $u$ itself, we will abuse notation and denote a pair $(u,\vec{g})\in QW^{1,p}(\Omega)$ by $\vec{\bf u} = (u,\nabla u)$. We will sometimes abuse notation even further, writing $u$ in place of $\vec{\bf u} = (u,\nabla u)$ for elements of $QW^{1,p}(\Omega)$.  As a result of this notational abuse, we will often write
$$\|\vec{\bf u}\|_{QW^{1,p}(\Omega)} = \|(u,\vec{g})\|_{QW^{1,p}(\Omega)} = \|u\|_{QW^{1,p}(\Omega)}. $$

\begin{rem}\label{densesobolev}
Using a simple density argument, one may show that if (H4) and (H5) hold, then \eqref{P} and \eqref{S} also hold for any $\vec{\bf w} = (w,\vec{h})\in QW^{1,p}(B)$ and $\vec{\bf u}=(u,\vec{g})\in QW^{1,p}_0(B)$ respectively, where $B\Subset\Omega$ is any metric ball.   That is, for any such pairs,
$$\frac{1}{\nu (B)}\int_B| w- w_B|d\nu\le C_0 r(B)\frac{1}{\mu(B)}\int_B|\sqrt{Q}\vec{h}|d\mu,\textrm{ and} $$
$$ \left(\frac{1}{\nu (B)}\int_B| u|^{kp}~d\nu\right)^{\frac{1}{kp}}\le C r(B)\left(\frac{1}{\mu(B)}\int_B|\sqrt Q~\vec{g}|^p~d\mu\right)^{\frac{1}{p}}.$$

We also explicitly note that if $\Omega_0$ is a compact domain such that $\Omega_0\Subset\Omega$ and if $\vec{\bf u}=(u,\vec{g})\in QW^{1,p}_0(\Omega_0)$, then extending $\vec{\bf u}=(u,\vec{g})$ to $0$ outside $\Omega_0$ we obtain an element of $QW^{1,p}_0(\Omega)$ which we will still denote by $\vec{\bf u}=(u,\vec{g})$ and which satisfies $$\|u\|_{QW^{1,p}(\Omega)}=\|u\|_{QW^{1,p}(\Omega_0)}.$$ In this case we will say that the support of such function is contained in $\Omega_0$.
\end{rem}

As a result of our definition, there is a natural product rule available for products of Lipschitz functions with compact support with elements of $QW^{1,p}(\Omega)$.  We present this in the following proposition and also note that it will play an important role in the proof of our main result.
\begin{pro}\label{prop1.6}
Assume condition (H3), let $u\in QW^{1,p}(\Omega)$, $\varphi\in Lip_0(\Omega)$, $\varphi \geq 0$, and $\theta\geq1$. Then $\varphi^\theta u\in QW_0^{1,p}(\Omega_0)$, where $\Omega_0$ is any bounded open set containing $\operatorname{supp}\varphi$ and such that $\Omega_0\Subset\Omega$,
with $$\sqrt{Q}\nabla(\varphi^\theta u)=\varphi^\theta\sqrt{Q}\nabla u+\theta\varphi^{\theta-1}u\sqrt{Q}\nabla\varphi\quad\text{a.e. in }\Omega.$$
\end{pro}
\begin{proof}
Let $u_j$ be a sequence of functions in $Lip_{\text{loc}}(\Omega)$ such that
\begin{align*}
   (i)~ & u_j\rightarrow u\quad\text{in }L^p(\Omega,d\nu)\text{ and }\nu-a.e.\\
    (ii)~ & \sqrt{Q}\nabla u_j\rightarrow \sqrt{Q}\nabla u \quad\text{in }(L^p(\Omega,d\mu))^n\text{ and }\mu-a.e.
\end{align*}
If $\Omega_0$ is any bounded open set containing $\operatorname{supp}\varphi$ and such that $\overline{\Omega}_0\subset\Omega$, then $v_j:=\varphi^\theta u_j\in Lip_0(\Omega_0)$. Since $$\int_{\Omega_0} |v_j-\varphi^\theta u|^p\,d\nu\leq C\int_\Omega|u_j-u|^p\,d\nu$$ we immediately have that $v_j$ converges to $\varphi^\theta u$ in $L^p(\Omega_0,d\nu)$ and $\nu$-a.e.

Since $v_j$ is differentiable Lebesgue-a.e. with
$$
\sqrt{Q}\nabla v_j=\varphi^\theta\sqrt{Q}\nabla u_j +\theta\varphi^{\theta-1}u_j\sqrt{Q}\nabla\varphi
$$
Lebesgue-a.e., by (H3) we see that
$$
\sqrt{Q}\nabla v_j\rightarrow \varphi^\theta\sqrt{Q}\nabla u +\theta\varphi^{\theta-1}u\sqrt{Q}\nabla\varphi
$$
$\mu$-a.e. Moreover, using (H3) again and recalling that $v=\|Q\|^{p/2}_{op}$
\begin{align*}
    &\int_{\Omega_0}|\sqrt{Q}\nabla v_j-(\varphi^\theta\sqrt{Q}\nabla u +\theta\varphi^{\theta-1}u\sqrt{Q}\nabla\varphi)|^p\,d\mu\\
    &\leq C\int_{\Omega_0}|\varphi^\theta(\sqrt{Q}\nabla u_j-\sqrt{Q}\nabla u)|^p\,d\mu +C\int_{\Omega_0}|\varphi^{\theta-1}(u_j-u)\sqrt{Q}\nabla\varphi|^p\,d\mu\\
    &\leq C\int_{\Omega_0}|\sqrt{Q}\nabla u_j-\sqrt{Q}\nabla u|^p\,d\mu +C\int_{\Omega_0}|u_j-u|^p\,vd\mu\\
    &\leq C\int_{\Omega_0}|\sqrt{Q}\nabla u_j-\sqrt{Q}\nabla u|^p\,d\mu +C\int_{\Omega_0}|u_j-u|^p\,d\nu.
\end{align*}
Hence $$\sqrt{Q}\nabla v_j\rightarrow \varphi^\theta\sqrt{Q}\nabla u +\theta\varphi^{\theta-1}u\sqrt{Q}\nabla\varphi$$
in $(L^p(\Omega_0,d\mu))^n$ and the proof is complete.
\end{proof}

The following propositions provide useful calculus results in the degenerate Sobolev space $QW^{1,p}(\Omega)$. Similar results can be found in Lemmas 4.1 and 4.2 in \cite{MRW1}, in case $\mu$ and $\nu$ are Lebesgue measure. Since the only change required here is the incorporation of the measures $\mu$ and $\nu$ and since the proofs go through taking into account the absolute continuity assumption in (H3), we will omit them.

\begin{pro}\label{appendix1}
Assume condition (H3). Let $(u, \nabla u) \in QW^{1,p}(\Omega)$ and $\phi\in C^1(\mathbb{R})$ with $\phi^\prime\in L^\infty(\mathbb{R})$. Then $(\phi(u),\nabla\phi(u))\in QW^{1,p}(\Omega)$ with
$\displaystyle{
\sqrt{Q}\nabla\big(\phi(u)\big)=\phi^\prime(u)\sqrt{Q}\nabla u,
\ \mu-a.e. in \ \Omega.
}$
\end{pro}

\begin{pro}\label{appendix2}
Assume condition (H3). Let $u\in QW^{1,p}(\Omega)$ and $k\in\R$. Then  we have $(\bar{u},\nabla\bar{u})\in QW^{1,p}(\Omega)$ with $\bar{u}=|u|+k$ and
\begin{equation}\label{26}
\sqrt{Q(x)}\nabla\bar{u}(x)=\begin{cases}
                               \,\sqrt{Q(x)}\nabla u(x)\,\,\,\qquad\text{if }u(x)\geq0,\\
                               -\sqrt{Q(x)}\nabla u(x)\qquad\text{if }u(x)<0
                            \end{cases}\quad\mu-\text{a.e. in }\Omega.
\end{equation}
\end{pro}

\begin{pro}\label{corHighInt}
Assume (H3), (H5), (H6). Then for every open subset $\Omega_0$ with $\Omega_0\Subset\Omega$ there exists $C>0$ such that
$$
\left(\int_{\Omega_0} |u|^{kp}\,d\nu\right)^\frac{1}{kp}\leq C\|u\|_{QW^{1,p}(\Omega)}
$$
for every $u\in QW^{1,p}(\Omega)$.
\end{pro}
\begin{proof}
We choose a small enough $r>0$ and use Lemma \ref{covering} to cover the set $\Omega_0$ with a finite number of balls $B_{\tau r}(x_1),\ldots, B_{\tau r}(x_K)$, with $\tau$ as in (H6). For each of the balls $B_{r}(x_1),\ldots, B_{r}(x_K),$ by (H6) we can find a Lipschitz function which is identically one on the ball of radius $\tau r$ and with support in the ball of radius $r$. Then we use (H5) in each of the balls and we conclude using Proposition \ref{prop1.6} with $\theta=1$ and (H3):
\begin{eqnarray*}
&&\left(\int_{\Omega_0}|u|^{kp}\,d\nu\right)^\frac{1}{kp}\\&&\leq\left(\sum_{j=1}^K\int_{B_{r}(x_j)}|\varphi_ju|^{kp}\,d\nu\right)^\frac{1}{kp}\\
&&\leq C\sum_{j=1}^K\left(\int_{B_{r}(x_j)}|\varphi_ju|^{kp}\,d\nu\right)^\frac{1}{kp}\\
&&\leq C\sum_{j=1}^K\left(\int_{B_{r}(x_j)}|\sqrt{Q}\nabla(\varphi_ju)|^{p}\,d\mu\right)^\frac{1}{p}\\
&&\leq C\sum_{j=1}^K\left(\int_{B_{r}(x_j)}|\varphi_j\sqrt{Q}\nabla u|^{p}\,d\mu+\int_{B_{r}(x_j)}|u\sqrt{Q}\nabla \varphi_j|^{p}\,d\mu\right)^\frac{1}{p}\\
&&\leq C\sum_{j=1}^K\left(\int_{B_{r}(x_j)}|\sqrt{Q}\nabla u|^{p}\,d\mu+\int_{B_{r}(x_j)}|u|^{p}v\,d\mu\right)^\frac{1}{p}\\
&&\leq C\sum_{j=1}^K\left(\int_{B_{r}(x_j)}|\sqrt{Q}\nabla u|^{p}\,d\mu+\int_{B_{r}(x_j)}|u|^{p}\,d\nu\right)^\frac{1}{p}\\
&&\leq C\|u\|_{QW^{1,p}(\Omega)}
\end{eqnarray*}
\end{proof}

\section{Fefferman-Phong embedding}\label{sec4}

The main tools to be exploited in our proofs of local boundedness and continuity of weak solutions to equations of the type \eqref{Qppoisson} are a Subrepresentaiton inequality and also a Fefferman-Phong inequality that we will develop now.  We recall the following
\begin{thm}[\cite{LW}, Theorem $A^*$]\label{rf}
If (H1), (H2), and (H4) hold, then for any ball $B\Subset\Omega$  one has
\begin{eqnarray}\label{subrep}
\left|u(x)-u_{B}\right| &\leq& C\displaystyle\int_B \frac{\rho(x,y)}{\mu(B(x,\rho(x,y)))}|\sqrt{Q}\nabla u(y)|\,d\mu(y)
\end{eqnarray}
for $d\nu$-a.e.  $x\in B$, for every $u\in QW^{1,p}(\Omega)$.
\end{thm}

\begin{rem}
Theorem \ref{rf} relies on the construction of special chains of metric balls, which can be carried out assuming condition (H1). Indeed one can show that the following property is a consequence of (H1), see Theorem D in \cite{LW}: let $B_0\Subset\Omega$ be a metric ball, then for any $x\in B_0$ there exists a chain of balls $\{B_k\}_{k\geq1}$ such that
\begin{enumerate}
    \item $B_k\subset B_0$ and $\rho(B_k,x)\rightarrow0$ as $k\rightarrow\infty$;
    \item $r(B_1)\approx r(B_0)$ and $r(B_k)\rightarrow0$ as $k\rightarrow\infty$
    \item if $y\in B_k$, then $\rho(x,y)\approx r(B_k)$
    \item $B_k\cap B_{k-1}$ contains a ball $S_k$ with $r(S_k)\approx r(B_k)\approx r(B_{k+1})$
    \item if $j<k$, then $B_k\subset cB_j$
    \item $\{B_k\}_{k\geq1}$ has bounded overlap, i.e. $\sum_k\chi_{B_k}(y)\leq c$ for all $y$
\end{enumerate}
where the constants of equivalence in (2), (3), (4) and the constants $c$ in (5) and (6) are independent of $x,k,j, B_0$, but the chain of balls $\{B_k\}_{k\geq1}$ depends on $x$.

We explicitly note that one could drop condition (H1), replacing it with the axiomatic assumption
\begin{itemize}
    \item[{\bf (H1$^\prime$)}] given any metric ball $B_0\Subset\Omega$, for every $x\in B_0$ there exists a chain of balls $\{B_k\}_{k\geq1}$ that satisfy conditions (1)--(6).
\end{itemize}
\end{rem}

This subrepresentation formula is useful in developing a weighted Fefferman-Phong inequality where we exploit the following class of functions.

\begin{defn}[\textit{Stummel-Kato}]\label{def1}
Let $V\in L^1_{loc}(\Omega)$, $p>1$, $r>0$ we set the function
$\Phi_V(r)$ as
$$
\sup_{x\in\Omega}\left(\int_{B(x,r)}
\!\!
\frac{\rho(x,y)}{\mu(B(x,\rho(x,y)))}\left(\int_{B(x,r)}
\!\!\!\!|V(z)|\frac{\rho(z,y) d\nu(z)}{\mu(B(z,\rho(z,y)))}\,\right)^\frac{1}{p-1}
\!\!\!\!d\mu(y)\right)^{p-1}
$$
We say that
\begin{enumerate}
    \item $V\in \tilde{M}_p(\Omega)$ if $\Phi_V(r)$ is finite for any $r>0$;
    \item $V\in M_p(\Omega)$ if $V\in \tilde{M}_p(\Omega)$ and  $\lim_{r\rightarrow0^+} \Phi_V(r)=0$
    \item $V\in M^\prime_p(\Omega)$ if $V\in M_p(\Omega)$ and  $$
    \int_0^\rho\frac{(\Phi_V(r))^\frac{1}{p-1}}{r}\,dr<\infty
    $$ for some $\rho>0$.
\end{enumerate}
\end{defn}

\begin{rem}
We remark that in the case $p=2$, $\mu=\nu=1$ and $Q=I$ we obtain the classical Stummel-Kato class introduced by  Aizenman and Simon in \cite{as} (see also \cite{SDH1} and \cite{SDH2}).
\end{rem}

\begin{rem}
Note that $M^\prime_p(\Omega)\subset M_p(\Omega)\subset\tilde{M}_p(\Omega)$ are vector spaces. If $V\in \tilde{M}_p(\Omega)\setminus\{0\}$ then $\Phi_V$ is defined in $(0,+\infty)$, it is positive, nondecreasing and continuous. If we set $\Phi_V(0)=\lim_{r\rightarrow0^+}\Phi_V(r)$ then the function is defined and continuous also at $0$. If $V\in M_p(\Omega)\setminus\{0\}$ then $\Phi_V(0)=0$ and we set
$$
\Phi^{-1}_V(t)=\min\{r\geq0\,|\,\Phi_V(r)=t\},
$$
for every $t\in[0,b)$, with $b=\sup\Phi_V$. Then $\Phi^{-1}_V$ is increasing, with
\begin{align*}
    &\Phi_V^{-1}(t)\leq r\qquad\forall r\geq0 \textrm{ such that }\Phi_V(r)=t,\\
    &\Phi_V(\Phi_V^{-1}(t))=t\qquad \forall t\in[0,b),\\
    &\Phi_V^{-1}(0)=0=\lim_{t\rightarrow0^+}\Phi_V^{-1}(t),\quad\Phi_V^{-1}(t)>0 \qquad \forall t\in(0,b).
\end{align*}
\end{rem}

\begin{thm}\label{feffermanphongembedding} (Fefferman-Phong Embedding)
Let $1<p<\infty$ and let $V$ belongs to $\tilde{M}_p(\Omega)$. Assume (H1), (H2) and (H4).  Then, there exists a constant $C>0$ (independent of $V$) so that for every metric ball $B \subset 4B\Subset\Omega$ one has
\begin{align}\label{feffermanphong}
\int_{B}\left|V(x)\right|~|u(x)-u_{B}|^p~d\nu &\leq C\Phi_V(2r)\int_{B}\left|\sqrt{Q}\nabla u\right|^p~d\mu
\end{align}
for any $u\in Lip(B)$, where $r=r(B)$.
\end{thm}
\begin{proof}
Fix a ball $B \subset 4B\Subset\Omega$.  For any $x,y\in \Omega, y\neq x$, set
$$R(x,y) = \frac{\rho(x,y)}{\mu(B(x,\rho(x,y)))}.$$
Using Theorem \ref{rf}, Fubini-Tonelli Theorem and H\"older inequality we get

\begin{align*}
&\int_{B}|V(x)|~|u(x)-u_{B}|^p~d\nu(x)\\
&\le C\int_{B}|V(x)||u(x)-u_{B}|^{p-1}\left(\int_B|\sqrt{Q}\nabla u(y)|R(x,y)~d\mu(y)\right)d\nu(x)\\
&= C\int_B|\sqrt{Q}\nabla u(y)|\left(\int_B|V(x)||u(x)-u_B|^{p-1}R(x,y)\,d\nu(x)\right)\,d\mu(y)\\
&\leq C\left(\int_B|\sqrt{Q}\nabla u(y)|^p\,d\mu(y)\right)^\frac{1}{p}\\
&\quad\quad\times\left[\int_B\left(\int_B|V(x)||u(x)-u_B|^{p-1}R(x,y)\,d\nu(x)\right)^\frac{p}{p-1}\,d\mu(y)\right]^\frac{p-1}{p}\\
&\leq C\left(\int_B|\sqrt{Q}\nabla u(y)|^p\,d\mu(y)\right)^\frac{1}{p}\\
&\quad\quad\times\left[\int_B\left(\int_B|V(x)||u(x)-u_B|^pR(x,y)\,d\nu(x)\right)\right.\\
&\hspace{1.65in}\times\left.\left(\int_B|V(z)|R(z,y)\,d\nu(z)\right)^\frac{1}{p-1}\,d\mu(y)\right]^\frac{p-1}{p}
\end{align*}
Reversing the order of integration again, we find this last line to be identical to
\begin{align*}&C\left(\int_B|\sqrt{Q}\nabla u(y)|^p\,d\mu(y)\right)^\frac{1}{p}\\
&\quad\quad\times\left[\int_B|V(x)||u(x)-u_B|^p\right.\\
&\quad\quad\quad\quad\times\left.\left(\int_B R(x,y)\left(\int_B|V(z)|R(z,y)\,d\nu(z)\right)^\frac{1}{p-1}\,d\mu(y)\right)\,d\nu(x)\right]^\frac{p-1}{p}.
\end{align*}
Before we apply Definition \ref{def1} to the last term above we note that the integration above is over the ball $B$ whose centre is of course not the same as $x$, the centre of $B(x,\rho(x,y))$ in the definition of $R(x,y)$.

However, using the swallowing property of metric balls, given any $x\in\Omega$, $B\subset 2B(x,r)\subset 4B\Subset\Omega$.  Using this idea, Definition \ref{def1} gives

\begin{align*}
&\int_{B}|V(x)|~|u(x)-u_{B}|^p~d\nu(x)\\
&\leq C\left(\int_B|\sqrt{Q}\nabla u(y)|^p\,d\mu(y)\right)^\frac{1}{p} \left[(\Phi_V(2r))^\frac{1}{p-1}\int_B|V(x)||u(x)-u_B|^p\,d\nu(x)\right]^\frac{p-1}{p}.
\end{align*}
Reordering terms then gives \eqref{feffermanphong} concluding the proof.
\end{proof}

\begin{cor}\label{sobnogain}
Let $1<p<\infty$ and let $V$ be any function in $\tilde{M}_p(\Omega)$. Assume (H1), (H2), (H4).  Then, there exists a constant $C>0$ (independent of $V$) such that for every metric ball $B$ with $8B\Subset\Omega$ one has
\begin{align}\label{Spc}
\int_{B}\left|V(x)\right||u(x)|^p~d\nu \leq C\Phi_V(4r)\int_{B}|\sqrt{Q}\nabla u|^p~d\mu
\end{align}
for any $u\in Lip_0(B)$, where $r=r(B)$.
\end{cor}
\begin{proof}
Let $u$ be a Lipschitz function compactly supported in  a ball $B$ with radius $r$.
By Theorem \ref{feffermanphongembedding} we have
\begin{equation}\label{11}
\begin{aligned}
&\left(\int_B |V(x)||u(x)|^pd\nu\right)^\frac{1}{p}\\
&\quad\leq \left(\int_B |V(x)||u(x)-u_{B}|^p~d\nu\right)^\frac{1}{p} + \left(\int_B |V(x)||u_{B}|^p~d\nu\right)^\frac{1}{p}\\
&\quad\leq C\left(\Phi_V(2r)\int_B |\sqrt{Q}\nabla u|^p~d\nu\right)^\frac{1}{p} + \left(\int_B |V(x)||u_{B}|^p~d\nu\right)^\frac{1}{p}
\end{aligned}
\end{equation}
Using that $u$ has compact support in $B$,
$$\left(1-\frac{\nu(B)}{\nu(2B)}\right)u_{B} = u_{B} - u_{2B}$$
and hence by Theorem \ref{feffermanphongembedding}
\begin{align*}
&\left(1-\frac{\nu(B)}{\nu(2B)}\right)\left(\int_B |V(x)||u_{B}|^p~d\nu\right)^\frac{1}{p}\\
&\quad= \left(\int_B |V(x)||u_{B}-u_{2B}|^p~d\nu\right)^\frac{1}{p}\\
&\quad\leq \left(\int_{B} |V(x)||u-u_{B}|^p~d\nu\right)^\frac{1}{p}+\left(\int_{2B} |V(x)||u-u_{2B}|^p~d\nu\right)^\frac{1}{p}\\
&\quad \leq C\left(\left(\Phi_V(2r)\right)^\frac{1}{p} +\left(\Phi_V(4r)\right)^\frac{1}{p}\right)\left( \int_{B} |\sqrt{Q}\nabla u|^p~d\mu\right)^\frac{1}{p}
\end{align*}
since $\operatorname{supp}u\subset B$. As $\frac{11}{2}B\subset8B\Subset\Omega$, Lemma \ref{2} provides a constant $c=c(A_\nu)>0$ such that $\left(1-\frac{\nu(B)}{\nu(2B)}\right) \geq c>0$. Dividing by this factor, since $V\in \tilde{M}_p(\Omega)$ and since $\Phi_V(r)$ is obviously nondecreasing, we obtain
\begin{equation}\label{13}
\left(\int_B |V(x)||u_B|^p~d\nu\right)^\frac{1}{p}\leq C \left(\Phi_V(4r) \int_{B} |\sqrt{Q}\nabla u|^p~d\mu,\right)^\frac{1}{p}.
\end{equation}
We get the result by \eqref{13}, \eqref{11} and monotonicity of $\Phi_V$.
\end{proof}

\begin{cor}\label{piccolograndestummel}
Let $1<p<\infty$ and let $V$ be any function in $M_p(\Omega)\setminus\{0\}$. Assume (H1), (H2), (H3), (H4), (H6). Then there exists $C>0$ such that for every small enough $\varepsilon >0$ and every Lipschitz function $u$ with compact support contained in $\Omega$ one has
\begin{align}\label{mexxopiccolopesostrong}
\int_{\Omega}|V|\,|u|^p\,d\nu
&\le \, C\varepsilon  \int_{\Omega}
|\sqrt Q\nabla u|^p \, d\mu+C\omega_V(\varepsilon)\int_{\Omega}|u|^p\,d\nu,
\end{align}
where $\omega_V(\varepsilon)=\varepsilon r_\varepsilon^{-\gamma p}$, $r_\varepsilon>0$ is any number such that $\Phi_V(r_\varepsilon)\leq\varepsilon$ and $\gamma>0$ is as in (H6).
\end{cor}

\begin{proof}
Since $V$ is not $0$ a.e. in $\Omega$, $b:=\sup\Phi_V>0$.
Let $\varepsilon\in(0,b)$ and $E$ be any bounded domain such that $\operatorname{supp} u\Subset E\Subset\Omega$. Let $r$ be a positive number to be chosen later. By Lemma \ref{covering} there exists $N=N(r)$ such that
$$
E\subset\bigcup_{j=1}^NB(x_j,\tau r)
$$
where $\tau$ is as in $(H6)$ and  $x_j\in E$ for all $j=1,\ldots,N.$
We also have
$$
\sum_{j=1}^N\chi_{B(x_j,r)}(x)\leq\left(\frac{18}{\tau}\right)^Q=M.
$$
From $(H6)$ there exists $\varphi_j \in Lip_0(B(x_j,r))$ such that $0\leq\varphi_j\leq1$, and  $\varphi_j=1$ on $B(x_j,\tau r)$, $j=1,\ldots,N.$ By Corollary \ref{sobnogain} we have
\begin{equation*}
  \begin{aligned}
  \int_{\Omega}V|u|^p\,d\nu&=\int_{\bigcup_{j=1}^NB(x_j,\tau r)}V|u|^p\,d\nu\\
  &\leq\sum_{j=1}^N\int_{B(x_j,r)}V|u|^p\varphi_j^p\,d\nu\\
  &\leq C\sum_{j=1}^N\Phi_V(4r)\int_{B(x_j,r)}|\sqrt{Q}\nabla(u\varphi_j)|^p\,d\mu\\
  &\leq C\Phi_V(4r)\left(\sum_{j=1}^N\int_{B(x_j,r)}|\sqrt{Q}\nabla u|^p\varphi_j^p\,d\mu\right.\\
  &\hspace{3.5cm}\left.+\sum_{j=1}^N\int_{B(x_j,r)}|\sqrt{Q}\nabla \varphi_j|^p|u|^p\,d\mu\right)
  \end{aligned}
\end{equation*}
for some positive constant $C$. Note that
\begin{equation*}
  \begin{aligned}
  \sum_{j=1}^N\int_{B(x_j,r)}|\sqrt{Q}\nabla u|^p\varphi_j^p\,d\mu&\leq \int_{\Omega}|\sqrt{Q}\nabla u|^p\sum_{j=1}^N\chi_{B(x_j,r)}\,d\mu\\
  &\leq M\int_{\Omega}|\sqrt{Q}\nabla u|^p\,d\mu.
\end{aligned}
\end{equation*}
By condition (H6)-(4),
\begin{equation*}
  \begin{aligned}
\sum_{j=1}^N\int_{B(x_j,r)}|\sqrt{Q}\nabla \varphi_j|^p|u|^p\,d\mu&\leq Cr^{-\gamma p}\sum_{j=1}^N\int_{B(x_j,r)}|u|^pv\,d\mu\\
&\leq  Cr^{-\gamma p}\int_{\Omega}|u|^pv\sum_{j=1}^N\chi_{B(x_j,r)}\,d\mu\\
&\leq  MCr^{-\gamma p}\int_{\Omega}|u|^pv\,d\mu.
\end{aligned}
\end{equation*}
Taking $t=4r$ we get
\begin{equation*}
  \begin{aligned}
 \int_{\Omega}|V||u|^p\,d\nu&\leq MC\Phi_V(t)\int_{\Omega}|\sqrt{Q}\nabla u|^p\,d\mu\\
  &\,\,\,\,\,+MC\Phi_V(t)Ct^{-\gamma p}\int_{\Omega}|u|^pv\,d\mu.
\end{aligned}
\end{equation*}
Let $t_\varepsilon$ be such that $\Phi_V(t_\varepsilon)\leq\varepsilon$. Condition (H3) yields, we get  \eqref{mexxopiccolopesostrong} where $\omega_V(\varepsilon)=\varepsilon t_\varepsilon^{-\gamma p}$.
\end{proof}

\begin{rem}\label{remmmm}
	Note that $r_\varepsilon=\Phi_V^{-1}(\varepsilon)$ in \eqref{mexxopiccolopesostrong} is an admissible choice.
\end{rem}
\begin{rem}\label{remmm}
	By a density argument and Fatou's Lemma, one can easily see that \eqref{mexxopiccolopesostrong} holds also for every function $u\in QW^{1,p}_0(E)$, where $E$ is any bounded domain such that $E\Subset\Omega$.
\end{rem}

\section{Local boundedness and Harnack inequality}\label{sec5}


In this section we prove local boundedness and continuity of weak solutions to equation \eqref{Qppoisson}, where $|f|^{p-1}$ belongs to the Stummel-Kato class (see Definition \ref{def1}).




\begin{defn}\label{weaksolutions}

A function $\vec{\bf u}=(u,\vec{g})\in QW^{1,p}(\Omega)$ is a solution of  \eqref{Qppoisson} if

\begin{equation}\label{eqn}
\int_{\Omega} |\sqrt Q~\vec{g}|^{p-2}\nabla\phi\cdot Q~\vec{g}~ h
dx+ \int_{\Omega}|f|^{p-2}f\phi ~mdx=0
\end{equation}
for every $\phi\in Lip_0(\Omega)$.
\end{defn}

We can enlarge the class of test functions in Definition \ref{weaksolutions} to include all pairs $(\phi,\vec{\psi})\in  QW^{1,p}_0(\Omega)$ such that $(\phi,\vec{\psi})\in  QW^{1,p}_0(\Omega_0)$ for some subdomain $\Omega_0 \Subset\Omega$ (such that both $\phi$ and $\vec{\psi}$ vanish outside of $\Omega_0$).

\begin{lem}\label{testexpand}  Let $Q(x)$ be symmetric, non negative definite matrix valued function for which $v(x) = \|Q(x)\|^{p/2}_{op} \in L^1_{loc}(\Omega)$ and assume (H3), (H5), (H6). Assume moreover that $f\in L^{p-\frac{p(k-1)}{pk-1}}_{\text{loc}}(\Omega,d\nu)$, where $k$ is as in (H5).  Let $(u,\vec{g})$ be a weak solution of \eqref{Qppoisson}. Then
\begin{align*}
\int_{\Omega} |\sqrt Q~\vec{g}|^{p-2}\vec{\psi}\cdot Q~\vec{g}~ h
dx+ \int_{\Omega}|f|^{p-2}f\varphi ~mdx=0
\end{align*}
for any compact subdomain $\Omega_0\Subset\Omega$ and any $(\varphi,\vec{\psi})\in QW^{1,p}_0(\Omega_0)$.
\end{lem}
\begin{proof}
For every $\varphi\in Lip_0(\Omega_0)$ \eqref{eqn} holds. By H\"{o}lder inequality
$$
\left|\int_{\Omega} |\sqrt Q~\vec{g}|^{p-2}\nabla\varphi\cdot Q~\vec{g}~ h
dx\right|\leq  \|u\|_{QW^{1,p}(\Omega)}^{p-1}\|\varphi\|_{QW^{1,p}(\Omega_0)}
$$
Proposition \ref{corHighInt} yields
\begin{align*}
\left|\int_{\Omega}|f|^{p-2}f\varphi ~mdx\right|&\leq \|f\|_{L^{p-\frac{p(k-1)}{pk-1}}(\Omega_0;d\nu)}^{p-1}\|\varphi\|_{L^{pk}(\Omega_0,d\nu)}\\
&\leq C\|f\|_{L^{p-\frac{p(k-1)}{pk-1}}(\Omega_0;d\nu)}^{p-1}\|\varphi\|_{QW^{1,p}(\Omega_0)}.
\end{align*}
Then the map
$$
(\varphi,\nabla\varphi)\mapsto\int_{\Omega} |\sqrt Q~\vec{g}|^{p-2}\nabla\varphi\cdot Q~\vec{g}~ h
dx+ \int_{\Omega}|f|^{p-2}f\varphi ~mdx
$$
is bounded in  $QW^{1,p}(\Omega_0)$-norm, and can thus be extended by continuity on $QW^{1,p}_0(\Omega_0)$. By density formula \eqref{eqn} holds for every $(\varphi,\vec{\psi})\in QW^{1,p}_0(\Omega_0)$.
\end{proof}


We also state here a Lemma, that we will need in the proof of our main Theorems (see \cite{RZ} and \cite{DZ}).

\begin{lem}\label{lemmaserie}
Let $\Phi:(0,\infty)\rightarrow(0,\infty)$ be a continuous nondecreasing function such that $\lim_{r\rightarrow0}\Phi(r)=0$, let $C>0$ be such that $C^{-1}\in \operatorname{Im}\Phi$ and let $k,p>1$. Then the series $$\sum_{j=0}^\infty\frac{1}{k^j}\log\Phi^{-1}\left(\frac{1}{Ck^{(p-1)j}}\right)$$
converges if and only if
$$\int_0^R \frac{(\Phi(r))^\frac{1}{p-1}}{r}\,dr$$
converges for some $R>0$. In this case, if we choose $R_0>0$ such that $\Phi(R_0)=C^{-1}$, we have
\begin{align*}
&\frac{1}{C^\frac{1}{p-1}}\left(-\frac{k-1}{k}\sum_{j=0}^\infty\frac{1}{k^j}\log\Phi^{-1}\left(\frac{1}{Ck^{(p-1)j}}\right)+\log R_0\right)\\
&\quad\leq \int_0^{R_0} \frac{(\Phi(r))^\frac{1}{p-1}}{r}\,dr\\
&\quad\quad\leq\frac{1}{C^\frac{1}{p-1}}\left(-(k-1)\sum_{j=0}^\infty\frac{1}{k^j}\log\Phi^{-1}\left(\frac{1}{Ck^{(p-1)j}}\right)+k\log R_0\right).
\end{align*}
\end{lem}

\begin{thm}[Local boundedness]\label{localboundedness}
Let $1<p<\infty$.  Assume (H1)-(H6). Suppose that $u\in QW^{1,p}(\Omega)$ is a solution of \eqref{Qppoisson} where $|f|^{p-1}\in M_p'(\Omega)$ and $f\in L^{p-\frac{p(k-1)}{pk-1}}_{\text{loc}}(\Omega,d\nu)$, where $k$ as in (H5). Then, there exists a constant $C>0$ such that for every ball $B\equiv B_r$ such that $8B\Subset\Omega$ we have
\begin{multline}\label{locale_limitatezza}
\|u\|_{L^\infty(B_{\tau r})}\leq \\
\leq C e^{\frac{\gamma k}{k-1}M_1}\!\! \left(r^{p(1-\gamma)}\frac{\nu(B_{r})}{\mu(B_{r})}\right)^{\frac{k}{p(k-1)}}\!\!\left(\left(\frac{1}{\nu(B_r)}\int_{B_r}|u|^p\,d\nu\right)^\frac{1}{p}\!\!\!+(\Phi_{|f|^{p-1}}(r))^{\frac{1}{p-1}}\right),
\end{multline}
where $\gamma$ is as in (H6) and $M_1>0$ is any constant satisfying
$$
\int_0^{r}\frac{\left(\Phi_{|f|^{p-1}}(s)\right)^{\frac{1}{p-1}}}{s}\,ds\leq M_1 \left(\Phi_{|f|^{p-1}}(r)\right)^{\frac{1}{p-1}}.
$$
\end{thm}

\begin{proof} We set $\bar{u}=|u| + \lambda$, with $\lambda>0$ arbitrary. Fix $q\ge 1$, $l>\lambda$ and let
\begin{equation*}
F(\bar{u})=
\begin{cases}
\bar{u}^q \qquad \qquad & \quad \text{if} \quad \lambda\le \bar{u}\le l \cr
ql^{q-1}(\bar{u} -l) + l^q & \quad \text{if} \quad \bar{u}\ge l\,.
\end{cases}
\end{equation*}
Set
\begin{equation*}
G(u)=\hbox{\rm sign}(u)\left(F(\bar{u})[F'(\bar{u})]^{p-1}-q^{p-1}\lambda^\beta
\right)
\quad
u\in\left(-\infty,+\infty\right)\,,
\end{equation*}
where $\beta$ satisfies  $pq=p+\beta-1$.

Let $B$ be a ball satisfying $8B\Subset\Omega$; for the rest of the proof, $r=r(B)$ will be fixed. Let $\phi=\varphi^pG(u)$,  where $\varphi$ is a Lipschitz function such that $0\le  \varphi \le 1$, $ \varphi\equiv 1$ in $B_{\tau r}$, compactly supported in $B_{r}$, see (H6). Arguing as in Subsection 2.2 in \cite{MRW1}, taking into account the presence of the measures $d\mu$ and $d\nu$, it is easy to see that there exists a sequence of $l$'s monotonically diverging at $+\infty$ such that the corresponding function $\phi\in QW^{1,p}_0(B_r)$ is a feasible test function in \eqref{eqn}. We will always and only be working with $l$'s in that sequence, but in order to keep the notation simple we will not specify the sequence when using $\phi$. Then using Young's inequality, and noting that $|G(u)|\leq F(\bar{u})(F'(\bar{u}))^{p-1}$, we obtain the following for any $\varepsilon\in (0,1)$
\begin{align}
\int_{B_{r}}& \varphi^pG^\prime(u)|\sqrt Q\nabla u|^p\, hdx \\
\nonumber&=\int_{B_{r}}  \varphi^pG^\prime  (u)|\sqrt Q\nabla u|^{p-2}\nabla u\cdot Q\nabla u\, hdx\\
\nonumber&=-p\int_{B_{r}}  \varphi^{p-1}G(u) |\sqrt Q\nabla u|^{p-2}\nabla  \varphi\cdot Q\nabla u\, hdx\\
\nonumber&\quad\quad\quad\quad\quad\quad\quad\;\;\,\,\,\,\,\,- \int_{B_{r}}|f|^{p-2}f \varphi^pG(u)\,mdx\\
\nonumber&\le p\int_{B_{r}}  \varphi^{p-1}|G(u)| |\sqrt Q\nabla u|^{p-1}|\sqrt Q\nabla  \varphi|\, hdx\\
\nonumber&\quad\quad\quad\quad\quad\quad\quad\;\;+\int_{B_{r}}|f|^{p-1} \varphi^p|G(u)|\,mdx\\
\nonumber&\le \varepsilon \int_{B_{r}} \varphi^p (F^\prime(\bar{u}))^p|\sqrt Q\nabla u|^p\, hdx +C(\varepsilon) \int_{B_{r}} \left(F(\bar{u})\right)^p  |\sqrt{Q}\nabla  \varphi|^p\,hdx\\
\nonumber&\quad\quad\quad\quad\quad\quad\quad\quad\quad\quad\quad\quad\quad\quad\;\;\;\;\;\,\,\,\,\,\,+\int_{B_{r}}|f|^{p-1} \varphi^p|G(u)|\,mdx
\end{align}
Then, set $\mathcal U=F(\bar{u})$ and $f_1=\left(\frac{|f|}{\lambda}\right)^{p-1}$, we get $|f|^{p-1}\le f_1\bar{u}^{p-1}$.

We have $f_1\in M'_p(\Omega)$ where
$$
\Phi_{f_1}(t)=\frac{1}{\lambda^{p-1}}\Phi_{|f|^{p-1}}(t)
\qquad
\forall t>0.
$$
By Propositions \ref{appendix1} and \ref{appendix2} we have $\mathcal{U}\in QW^{1,p}(\Omega)$. Note that, since $$F'(\bar{u})^p\le G'(u)\le p (F'(\bar{u}))^p,$$
we can absorb the first term in the last line into the first term in the first line. Moreover we also have
$$
|G(u)|\leq F(\bar{u})(F'(\bar{u}))^{p-1},\qquad\bar{u}F'(\bar{u})\leq qF(\bar{u}),
$$
and then we obtain
\begin{align}\label{14}
\int_{B_{r}} &\varphi^p|\sqrt Q\nabla \mathcal U|^p hdx\\
\nonumber&\le C \int_{B_{r}} |\mathcal U|^p  |\sqrt Q\nabla  \varphi|^p\,hdx+C\int_{B_{r}}  \varphi^pf_1|\bar{u}|^{p-1}(F'(\bar{u}))^{p-1}\mathcal{U}\,mdx \\
\nonumber&\le C \int_{B_{r}} |\mathcal U|^p  |\sqrt Q\nabla  \varphi|^p\,hdx+Cq^{p-1}\int_{B_{r}}  \varphi^pf_1 |\mathcal U|^p\,mdx.
\end{align}
By Corollary \ref{piccolograndestummel} on $\varphi \, \mathcal{U}$ (with the choice $r_\varepsilon=(\Phi_{f_1}^{-1}(\varepsilon))$), see also Remarks \ref{remmmm} and \ref{remmm}, and we obtain
\begin{align*}
\int_{B_{r}} \varphi^p&|\sqrt Q\nabla \mathcal U|^p hdx\\
&\le C \int_{B_{r}} |\mathcal U|^p  |\sqrt Q\nabla  \varphi|^phdx+Cq^{p-1}\varepsilon \int_{B_{r}} |\sqrt Q \nabla ( \varphi \mathcal U)| ^ph\,dx\\
&\qquad\qquad\qquad\qquad\qquad\qquad+ Cq^{p-1}\omega_{f_1}(\varepsilon)  \int_{B_{r}} | \varphi \mathcal U|^p m\,dx\\
&\le C(1+q^{p-1}\varepsilon)\!\!\! \int_{B_{r}} |\mathcal U|^p  |\sqrt Q\nabla  \varphi|^phdx+\hat{C}q^{p-1}\varepsilon\!\!\! \int_{B_{r}}  \varphi^p|\sqrt Q \nabla  \mathcal U|^phdx\\
&\qquad\qquad\qquad\qquad\qquad\qquad+ Cq^{p-1}\omega_{f_1}(\varepsilon) \int_{B_{r}} | \varphi \mathcal U|^p m\,dx.
\end{align*}
for every $0<\varepsilon<\sup\Phi_{f_1}$. Since we are assuming $f$ is not $0$ a.e. in $\Omega$, we choose $\lambda=(2\hat{C}\Phi_{|f|^{p-1}}(r))^{\frac{1}{p-1}}>0$. In particular $\Phi_{f_1}(r)=\frac{1}{2\hat{C}}$ and $\sup\Phi_{f_1}>\frac{1}{2\hat{C}}$.
We can choose $\varepsilon=\frac{1}{2\hat{C}q^{p-1}}$ and then
\begin{align}\label{51}
&\int_{B_{r}} \varphi^p|\sqrt Q\nabla \mathcal U|^p hdx\\
\nonumber&\quad\le C \int_{B_{r}} |\mathcal U|^p  |\sqrt Q\nabla  \varphi|^phdx
+Cq^{p-1}\omega_{f_1}\left(\frac{1}{2\hat{C}q^{p-1}}\right)\int_{B_{r}} | \varphi \mathcal U|^p m\,dx.
\end{align}
Since $d\nu=mdx$, $d\mu=hdx$,  $vh\leq Cm$, by Remark \ref{densesobolev} and \eqref{51} we have
\begin{align*}
&\left( \int_{B_{r}}| \varphi \mathcal U|^{kp}d\nu\right)^{\frac{1}{k}}\\
&\le C r^p \frac{(\nu(B_{r}))^\frac{1}{k}}{\mu(B_{r})}\int_{B_{r}}|\sqrt Q \nabla( \varphi\mathcal U)|^pd\mu \\
&\le C r^p\frac{(\nu(B_{r}))^\frac{1}{k}}{\mu(B_{r})}\int_{B_{r}}|\mathcal U|^p  |\sqrt Q\nabla  \varphi|^pd\mu\\
&\qquad\qquad+Cq^{p-1}\omega_{f_1}\left(\frac{1}{2\hat{C}q^{p-1}}\right)r^p\frac{(\nu(B_{r}))^\frac{1}{k}}{\mu(B_{r})}\int_{B_{r}} | \varphi \mathcal U|^p d\nu.
\end{align*}
We specialize our choice of test function $\varphi$, using the sequence of Lipschitz cutoff functions $\{\varphi_j\}$ provided by (H6), relative to the ball $B$. We denote by $S_j:=\operatorname{supp}\varphi_j$, $S_0:=B$.
\begin{align*}
&\left( \int_{S_{j+1}}| \mathcal U|^{kp}d\nu\right)^{\frac{1}{k}}\le\left( \int_{S_j}| \varphi_j \mathcal U|^{kp}d\nu\right)^{\frac{1}{k}}\\
&\le C r^{p(1-\gamma)}\frac{(\nu(B_{r}))^\frac{1}{k}}{\mu(B_{r})}T^{pj}\int_{S_j}|\mathcal U|^p\,vd\mu \\
&\qquad\qquad+Cq^{p-1}\omega_{f_1}\left(\frac{1}{2\hat{C}q^{p-1}}\right) r^p\frac{(\nu(B_{r}))^\frac{1}{k}}{\mu(B_{r})}\int_{S_j}| \mathcal U|^pd\nu \\
&\le Cr^p\frac{(\nu(B_{r}))^\frac{1}{k}}{\mu(B_{r})}\left(r^{-p\gamma}T^{pj}+q^{p-1}\omega_{f_1}\left(\frac{1}{2\hat{C}q^{p-1}}\right)\right) \int_{S_j}| \mathcal U|^pd\nu
\end{align*}
Letting $l\to \infty$ along the sequence of $l$'s
we are using, by monotone convergence we have
\begin{align}\label{nnnnnn}
&\left( \int_{S_{j+1}}|\bar{u}|^{kqp}d\nu\right)^{\frac{1}{kqp}} \\
\nonumber&\le C^{\frac{1}{pq}}r^{\frac{1-\gamma}{q}}\frac{(\nu(B_{r}))^{\frac{1}{kqp}}}{(\mu(B_{r}))^{\frac{1}{qp}}}\left(T^{pj}+r^{p\gamma}q^{p-1}\omega_{f_1}\left(\frac{1}{2\hat{C}q^{p-1}}\right)\right)^{\frac{1}{qp}}\left( \int_{S_j}|\bar{u}|^{qp}d\nu\right)^{\frac{1}{qp}}.
\end{align}
Since $\Phi_{f_1}(r)=\frac{1}{2\hat{C}}$ and $\Phi_{f_1}^{-1}$ is monotone we have
$$r^{-1}\Phi_{f_1}^{-1}\left(\frac{1}{2\hat{C}q^{p-1}}\right)\leq r^{-1}\Phi_{f_1}^{-1}\left(\frac{1}{2\hat{C}}\right)\leq 1$$
and
$$r^{p\gamma}q^{p-1}\omega_{f_1}\left(\frac{1}{2\hat{C}q^{p-1}}\right)=
r^{p\gamma}\frac{1}{2\hat{C}}\left(\Phi_{f_1}^{-1}\left(\frac{1}{2\hat{C}q^{p-1}}\right)\right)^{-\gamma p}\geq\frac{1}{2\hat{C}}.$$
Thus we have
\begin{align*}
&\|\bar{u}\|_{L^{kpq}_\nu(S_{j+1})}\\
&\le C^{\frac{1}{pq}}r^{\frac{1}{q}}\frac{(\nu(B_{r}))^{\frac{1}{kqp}}}{(\mu(B_{r}))^{\frac{1}{pq}}}T^{\frac{j}{q}}
\left(\Phi^{-1}_{f_1}\left(\frac{1}{2\hat{C}q^{p-1}}\right)\right)^{-\frac{\gamma}{q}}\|\bar{u}\|_{L^{pq}_\nu(S_j)}.
\end{align*}

We choose $q=q_j=k^j$, set $\alpha_j=pk^j$ and we obtain
$$\|\bar{u}\|_{L^{\alpha_{j+1}}_\nu(S_{j+1})}\le (Cr^{p})^{\frac{1}{\alpha_j}}\frac{(\nu(B_{r}))^{\frac{1}{k\alpha_{j}}}}{(\mu(B_{r}))^{\frac{1}{\alpha_j}}} T^{\frac{pj}{\alpha_j}}\left(\Phi^{-1}_{f_1}\left(\frac{p^{p-1}}{2\hat{C}\alpha_j^{p-1}}\right)\right)^{-\frac{p\gamma}{\alpha_j}} \|\bar{u}\|_{L^{\alpha_j}_\nu(S_j)}.
$$
We iterate this inequality, sending $j$ to $\infty$, and we find
\begin{align}
\label{nnnn}&\|\bar{u}\|_{L^\infty(B_{\tau r})}\\
\nonumber&\le Cr^{\frac{k}{k-1}}\frac{(\nu(B_{r}))^{\frac{1}{p(k-1)}}}{(\mu(B_{r}))^{\frac{k}{p(k-1)}}}\prod_{j=0}^{\infty}\left(\Phi^{-1}_{f_1}\left(\frac{p^{p-1}}{2\hat{C}\alpha_j^{p-1}}\right)\right)^{-\frac{p\gamma}{\alpha_j}}\|\bar{u}\|_{L^p_\nu(B_{r})}\\
\nonumber&=C\left(\exp\left(\frac{1}{(2\hat{C})^\frac{1}{p-1}}\sum_{j=0}^{\infty}\frac{-p\gamma}{\alpha_j}\log\left(\Phi^{-1}_{f_1}\left(\frac{p^{p-1}}{2\hat{C}\alpha_j^{p-1}}\right)\right)\right)\right)^{(2\hat{C})^\frac{1}{p-1}} \times \\
\nonumber&\quad\quad \times \left(r^{p}\frac{\nu(B_{r})}{\mu(B_{r})}\right)^{\frac{k}{p(k-1)}}\left(\frac{1}{\nu(B_r)}\int_{B_r}\bar{u}^p\,d\nu\right)^\frac{1}{p}.
\end{align}
Now we note that
\begin{align*}
    &\frac{1}{(2\hat{C})^\frac{1}{p-1}}\sum_{j=0}^{\infty}\frac{-p\gamma}{\alpha_j}\log\left(\Phi^{-1}_{f_1}\left(\frac{p^{p-1}}{2\hat{C}\alpha_j^{p-1}}\right)\right)\\
    &\quad=-\gamma\frac{1}{(2\hat{C})^\frac{1}{p-1}}\sum_{j=0}^{\infty}\frac{1}{k^j}\log\left(\Phi^{-1}_{f_1}\left(\frac{1}{2\hat{C}k^{(p-1)j}}\right)\right)
\end{align*}
and hence the series converges by Lemma \ref{lemmaserie} since $\Phi_{f_1}(r)=\frac{1}{2\hat{C}}$ and thus
\begin{align*}
&-\gamma\frac{1}{(2\hat{C})^\frac{1}{p-1}}\sum_{j=0}^{\infty}\frac{1}{k^j}\log\left(\Phi^{-1}_{f_1}\left(\frac{1}{2\hat{C}k^{(p-1)j}}\right)\right)\\
&\qquad\leq \frac{\gamma k}{k-1}\left(\int_0^{r}\frac{\left(\Phi_{f_1}(s)\right)^{\frac{1}{p-1}}}{s}\,ds-\frac{1}{(2\hat{C})^\frac{1}{p-1}}\log r\right)\\
&\qquad=\frac{\gamma k}{k-1}\frac{1}{(2\hat{C})^\frac{1}{p-1}}\left(\frac{1}{\left(\Phi_{|f|^{p-1}}(r)\right)^{\frac{1}{p-1}}}\int_0^{r}\frac{\left(\Phi_{|f|^{p-1}}(s)\right)^{\frac{1}{p-1}}}{s}\,ds-\log r\right)<\infty
\end{align*}
by our definition of $f_1$. Now if
$$
\frac{1}{\left(\Phi_{|f|^{p-1}}(r)\right)^{\frac{1}{p-1}}}\int_0^{r}\frac{\left(\Phi_{|f|^{p-1}}(s)\right)^{\frac{1}{p-1}}}{s}\,ds\leq M_1
$$
from \eqref{nnnn} we have
\begin{equation}\label{2323}
\|\bar{u}\|_{L^\infty(B_{\tau r})}\leq C e^{\frac{\gamma k}{k-1}M_1} \left(r^{p(1-\gamma)}\frac{\nu(B_{r})}{\mu(B_{r})}\right)^{\frac{k}{p(k-1)}}\left(\frac{1}{\nu(B_r)}\int_{B_r}\bar{u}^p\,d\nu\right)^\frac{1}{p}.
\end{equation}
This concludes the proof if $\lambda=(\Phi_{|f|^{p-1}}(r))^{\frac{1}{p-1}}>0$, i.e. if $f$ does not vanish almost everywhere. If $(\Phi_{|f|^{p-1}}(r))^{\frac{1}{p-1}}=0$ then $f_1=0$ a.e. and \eqref{14} yields
\begin{align*}
\int_{B_{r}} \varphi^p|\sqrt Q\nabla \mathcal U|^p hdx\le C \int_{B_{r}} |\mathcal U|^p  |\sqrt Q\nabla  \varphi|^p\,hdx.
\end{align*}
Proceeding as in the previous case, for any $\lambda>0$ one gets

\begin{align*}
\left( \int_{S_{j+1}}|\bar{u}|^{kqp}d\nu\right)^{\frac{1}{kqp}}\le C^{\frac{1}{pq}}r^{\frac{1-\gamma}{q}}\frac{(\nu(B_{r}))^{\frac{1}{kqp}}}{(\mu(B_{r}))^{\frac{1}{qp}}}T^{\frac{j}{q}}\left( \int_{S_j}|\bar{u}|^{qp}d\nu\right)^{\frac{1}{qp}},
\end{align*}
formally as in \eqref{nnnnnn} with $\omega_{f_1}=0$. Choosing again $q=q_j=k^j$ and setting $\alpha_j=pk^j$ we obtain
$$\|\bar{u}\|_{L^{\alpha_{j+1}}_\nu(S_{j+1})}\le \left(Cr^{p(1-\gamma)}\right)^{\frac{1}{\alpha_j}}\frac{(\nu(B_{r}))^{\frac{1}{k\alpha_{j}}}}{(\mu(B_{r}))^{\frac{1}{\alpha_j}}} T^{\frac{pj}{\alpha_j}} \|\bar{u}\|_{L^{\alpha_j}_\nu(S_j)}.
$$
Iterating the above inequality and letting $j$ to infinity we deduce that
\begin{align}
\label{n}&\|\bar{u}\|_{L^\infty(B_{\tau r})}\\
\nonumber&\le Cr^{\frac{k(1-\gamma)}{k-1}}\frac{(\nu(B_{r}))^{\frac{1}{p(k-1)}}}{(\mu(B_{r}))^{\frac{k}{p(k-1)}}}\|\bar{u}\|_{L^p_\nu(B_{r})}\\
\nonumber&=C\left(r^{p(1-\gamma)}\frac{\nu(B_{r})}{\mu(B_{r})}\right)^{\frac{k}{p(k-1)}}\left(\frac{1}{\nu(B_r)}\int_{B_r}\bar{u}^p\,d\nu\right)^\frac{1}{p}
\end{align}
and sending $\lambda$ to $0$ we conclude.
\end{proof}

\begin{rem}
	If $\nu\leq C\mu$ for some constant $C>0$ and $\gamma=1$, then constant in \eqref{locale_limitatezza} does not depend on the metric balls but it depends on $r$ and $f$ only through $M_1$.
\end{rem}

\begin{cor}
Let $1<p<\infty$.  Assume conditions (H1)-(H6). Suppose that $u\in QW^{1,p}(\Omega)$ is a solution of \eqref{Qppoisson} where $|f|^{p-1}\in M_p'(\Omega)$ and $f\in L^{p-\frac{p(k-1)}{pk-1}}_{\text{loc}}(\Omega,d\nu)$, with $k$ as in (H5). Let $E,E_0$ be bounded open domains such that $E\Subset E_0\Subset\Omega$ and let
$0<r<\frac{1}{8}\rho(E,\partial E_0)$ be such that $\Phi_{|f|^{p-1}}(r)<\sup \Phi_{|f|^{p-1}}$, if $f$ does not vanish almost everywhere. Then there exists $C>0$ depending on $E,E_0,r$ such that
\begin{equation*}
\|u\|_{L^{\infty}(E)}
\le
Ce^{\frac{\gamma k}{k-1}M_1}\left\{\left(\int_{E_0}|u|^p\,d\nu\right)^
{\frac{1}{p}}+(\Phi_{|f|^{p-1}}(r))^{\frac{1}{p-1}}\right\},
\end{equation*}
where $M_1>0$ is any constant such that
$$
\int_0^{r}\frac{\left(\Phi_{|f|^{p-1}}(s)\right)^{\frac{1}{p-1}}}{s}\,ds\leq M_1 \left(\Phi_{|f|^{p-1}}(r)\right)^{\frac{1}{p-1}}.
$$
 \end{cor}
\begin{proof}
Covering $E$ with a finite number of balls $B_{\tau r}(x_1)$, $\ldots, B_{\tau r}(x_N)$ for suitably small $r>0$, with centers in $E$ and such that the closures of $B_{r}(x_1),\ldots, B_{r}(x_N)$ are contained in $E_0$, the result follows by  applying Theorem \ref{localboundedness} on each $B_{\tau r}(x_j)$.
\end{proof}
\begin{thm}\label{lastbutone}
Let $1<p<\infty$, (H1)-(H6) hold true and $u\in QW^{1,p}(\Omega)$ be a non negative solution of \eqref{Qppoisson}, where $|f|^{p-1}\in M_p'(\Omega)$, $f\in L^{p-\frac{p(k-1)}{pk-1}}_{\text{loc}}(\Omega,d\nu)$ and $k$ is as in (H5).

Assume there exist $r^*,C>0$ such that
\begin{equation}\label{3456}
\sup r^{p(1-\gamma)}\frac{\nu(B_r)}{\mu(B_{r})}\leq C
\end{equation}
where the supremum is computed over all balls such that ${8}B\Subset\Omega$ and $r\leq r^*$, with $\gamma$ as in (H6).

Then, there exist constants $C_2,C_3>0$ such that for every ball $B=B_r$ with $\frac{8}{\tau}B\Subset\Omega$, $0<r\leq r^*$ such that $\Phi_{|f|^{p-1}}\left(\frac{r}{\tau}\right)<\sup\Phi_{|f|^{p-1}}$ if $f$ does not vanish a.e. we have
\begin{align*}
\sup_{B_{\tau r}} u\leq C_2 e^{C_3M_1} \left(\inf_{B_{\tau r}}u+(\Phi_{|f|^{p-1}}(r))^{\frac{1}{p-1}}\right)
\end{align*}
where $\tau$ is as in (H6) and $M_1>0$ is any constant satisfying
\begin{equation}\label{12356}
\int_0^{r}\frac{\left(\Phi_{|f|^{p-1}}(s)\right)^{\frac{1}{p-1}}}{s}\,ds\leq M_1 \left(\Phi_{|f|^{p-1}}(r)\right)^{\frac{1}{p-1}}.
\end{equation}
\end{thm}

\begin{proof}
We start as in Theorem \ref{localboundedness}, setting $\bar{u}=u+\lambda$, where $\lambda>0$ is specified below. We note that we can apply Theorem \ref{localboundedness} on $B_\frac{r}{\tau}$, hence $\|u\|_{L^\infty(B_r)}$ is finite. Now let $\varphi$ be a non negative Lipschitz function compactly supported in $B_{r}$, then for any $\beta\in\R\setminus\{0\}$ we can take $\phi= \varphi^p\bar{u}^\beta$ as a test function in \eqref{eqn}.

We find
\begin{align}\label{55}
\int_{B_{r}}\varphi ^p\bar{u}^{\beta-1}|\sqrt Q\nabla u|^pd\mu&\leq C\left\{|\beta|^{-p}\int_{B_{r}} \bar{u}^{\beta+p-1}|\sqrt Q\nabla \varphi|^pd\mu\right.\\
\nonumber&\qquad\,\,\,\,\left.+|\beta|^{-1}\int_{B_r} f_1\varphi^p\bar{u}^{\beta+p-1}\, d\nu\right\},
\end{align}
where we set $f_1=\left(\frac{|f|}{\lambda}\right)^{p-1}$, such that we have $|f|^{p-1}\le f_1\bar{u}^{p-1}$ and  $f_1\in M'_p (\Omega)$.  Now set
\begin{align*}
&\U(x) =
\begin{cases}
\bar{u}^q(x)\quad
\quad \hbox{where } pq=p+\beta-1, \hbox{ if } \beta\neq 1-p
\\
\log \bar{u}(x)
\quad \hbox{if } \beta=1-p
\end{cases}
\end{align*}
Then from \eqref{55} we have
\begin{align} \label{eq:betadiverso}
&\int_{B_{r}}\varphi^p|\sqrt Q\nabla \U|^p\,d\mu\\
&\le
\nonumber C|q|^p
\left\{|\beta|^{-p}\int_{B_{r}}|\sqrt Q\nabla  \varphi|^p\U^p \,d\mu +|\beta|^{-1}\int_{B_{r}}f_1\varphi^p \U^p\,d\nu\right\}\,
\quad{\rm if }\beta \neq 1-p,
\end{align}
while
\begin{multline} \label{eq:betauguale}
\int_{B_{r}}\varphi^p|\sqrt Q \nabla\U|^p\,d\mu\\
\le C\left\{\int_{B_{r}}|\sqrt Q\nabla\varphi|^p\,d\mu
+\int_{B_{r}}f_1\varphi^p\,d\nu\right\}\quad \hbox{if} \ \beta
=1-p\,.
\end{multline}
From now on we consider the case when $f$ is not $0$ almost everywhere, such that $\Phi_{|f|^{p-1}}(r)>0$. We set
\begin{equation}\label{lambdaaa}
\lambda=(2\sigma\Phi_{|f|^{p-1}}(r))^\frac{1}{p-1},
\end{equation}
where $\sigma\geq1$ is a suitable constant to be chosen later. In particular $\Phi_{f_1}(r)=\frac{1}{2\sigma}$.

We start considering \eqref{eq:betauguale}. We use Corollary \ref{piccolograndestummel}, with the choice $\varepsilon=\frac{1}{2\sigma}$ and $r_\varepsilon=r$, and by condition (H5) we have
\begin{align*}
\int_{B_{r}}f_1\varphi^p\,d\nu&\le \frac{C}{\sigma} \int_{B_{r}}|\sqrt Q\nabla\varphi|^p\,d\mu+ \frac{C}{\sigma} r^{-\gamma p}\int_{B_r}\varphi^p\,d\nu\\
&\leq C \int_{B_{r}}|\sqrt Q\nabla\varphi|^p\,d\mu+Cr^{-\gamma p}\nu(B_r) \left(\frac{1}{\nu(B_r)}\int_{B_r}\varphi^{kp}\,d\nu\right)^\frac{1}{k}\\
&\leq C\left(1+r^{p(1-\gamma)}\frac{\nu(B_r)}{\mu(B_r)}\right)\int_{B_{r}}|\sqrt Q\nabla\varphi|^p\,d\mu
\end{align*}
Then, if $\beta=1-p$
\begin{equation}\label{52}
\int_{B_{r}} \varphi^p|\sqrt Q\nabla \mathcal U|^p d\mu\le C\left(1+r^{p(1-\gamma)}\frac{\nu(B_r)}{\mu(B_r)}\right) \int_{B_{r}}|\sqrt Q\nabla\varphi|^pd\mu.
\end{equation}
Now we choose $\varphi$ to be the first function in the sequence provided by condition (H6). From (H4) and \eqref{52} we get
\begin{align*}
& \left(\frac{1}{\nu(B_{\tau r})}\int_{B_{\tau r}}|\mathcal U-\mathcal U_{B_{\tau r};\nu}|\,d\nu\right)^p\\
&\qquad\le C\left(\frac{r}{\mu(B_{\tau r})}\int_{B_{\tau r}}|\sqrt Q\nabla \mathcal U|\,d\mu\right)^p\\
&\qquad\le C\frac{r^p}{\mu(B_{\tau r})} \int_{B_{r}}\varphi^p|\sqrt Q\nabla \mathcal U|^p \,d\mu\\
&\qquad\le C\frac{r^p}{\mu(B_{\tau r})}\left(1+r^{p(1-\gamma)}\frac{\nu(B_r)}{\mu(B_r)}\right) \int_{B_{r}}|\sqrt Q\nabla\varphi|^pd\mu\\
&\qquad\le C\frac{r^{p(1-\gamma)}}{\mu(B_{\tau r})}\left(1+r^{p(1-\gamma)}\frac{\nu(B_r)}{\mu(B_r)}\right) \int_{B_{r}}d\nu\\
&\qquad= Cr^{p(1-\gamma)}\frac{\nu(B_r)}{\mu(B_{\tau r})}\left(1+r^{p(1-\gamma)}\frac{\nu(B_r)}{\mu(B_r)}\right).
\end{align*}
By the doubling assumption on $\nu$ and by condition \eqref{3456} we have that
$$
\sup\left(\frac{1}{\nu(B_{s})}\int_{B_{s}}|\mathcal U-\mathcal U_{B_{s};\nu}|\,d\nu\right)
$$
is finite, where the supremum is computed over all balls with radius small enough and such that $\frac{8}{\tau}B_s\Subset\Omega$. Then $\mathcal U$ locally belongs to $BMO_\nu$.  Thus, arguing as in \cite{MRW2}[Proposition 5.1 and Corollary 5.3], we see that there are positive constants $p_0,C$, independent of $\sigma\geq1$, such that for $R$ sufficiently small and $\frac{8}{\tau}B_R\Subset\Omega$ the John-Nirenberg inequality
\begin{equation}\label{johnnirenberg}
\left(\frac{1}{\nu(B_{R})}\int_{B_{R}}e^{p_0 \mathcal U}\,d\nu\right) \left(\frac{1}{\nu(B_{R})}\int_{B_{R}}
e^{-{p_0 \mathcal U}}\,d\nu\right)  \le C\,,
\end{equation}
holds, where we recall that here $\mathcal{U}=\log\bar{u}$. We explicitly note for later use that $p_0>0$ small can be chosen independent of $\sigma\geq1$ and in such a way that
\begin{equation}\label{1379}
p_0k^j-p+1\neq0
\end{equation}
for every $j\in\mathbb{N}$, with $k$ given by (H5). In particular we have
$$\min_{j\in\mathbb{N}} |p_0k^j-p+1|>0.$$
Now we consider the case $\beta> 1-p$, see \eqref{eq:betadiverso}. For every $j\in\mathbb{N}$ we choose $\beta_j=p_0k^j-p+1$ such that $q_j=\frac{\beta_j+p-1}{p}=\frac{p_0}{p}k^j>0$. Using \eqref{eq:betadiverso} and Corollary \ref{piccolograndestummel} we obtain for every $0<\varepsilon<\sup\Phi_{f_1}$
\begin{align*}
\int_{B_{r}}&\varphi^p|\sqrt Q\nabla \U|^p\,d\mu\\
\nonumber&\leq C|q_j|^p|\beta_j|^{-p}\int_{B_{r}}|\sqrt Q\nabla  \varphi|^p\U^p \,d\mu  +C|q_j|^{p}|\beta_j|^{-1}\varepsilon\int_{B_{r}}\varphi^p|\sqrt Q\nabla \U|^p\,d\mu\\
&\,\,\,\,+C|q_j|^{p}|\beta_j|^{-1}\varepsilon\int_{B_{r}}|\sqrt{Q}\nabla\varphi|^p \U^p\,d\mu+C|q_j|^p|\beta_j|^{-1}\omega_{f_1}(\varepsilon)\int_{B_r}\varphi^p\mathcal{U}^p\,d\nu.
\end{align*}
By our choice of $\beta_j,q_j$ and by condition \eqref{1379} we see that
\begin{align}
\label{pigreco}\int_{B_{r}}&\varphi^p|\sqrt Q\nabla \U|^p\,d\mu\\
\nonumber&\leq C\int_{B_{r}}|\sqrt Q\nabla  \varphi|^p\U^p \,d\mu  +\tilde{C}k^{j(p-1)}\varepsilon\int_{B_{r}}\varphi^p|\sqrt Q\nabla \U|^p\,d\mu\\
\nonumber&\,\,\,\,+\tilde{C}k^{j(p-1)}\varepsilon\int_{B_{r}}|\sqrt{Q}\nabla\varphi|^p \U^p\,d\mu+Ck^{j(p-1)}\omega_{f_1}(\varepsilon)\int_{B_r}\varphi^p\mathcal{U}^p\,d\nu.
\end{align}
Up to choosing an even larger constant, we can assume for later use that $\tilde{C}\geq1$. Now we choose $\varepsilon=\varepsilon_j=\frac{1}{2\tilde{C}k^{j(p-1)}}$, and we note that by our assumptions we can fix the constant $\sigma\geq1$ introduced in \eqref{lambdaaa} and set $\sigma=\tilde{C}$. With this choice of $\sigma$ we have $\Phi_{f_1}(r)=\frac{1}{2\sigma}\geq\varepsilon,$ and hence $\varepsilon$ is a feasible choice in \eqref{pigreco}. Thus we have
\begin{align*}
    \int_{B_{r}}&\varphi^p|\sqrt Q\nabla \U|^p\,d\mu\\
&\leq C\int_{B_{r}}|\sqrt Q\nabla  \varphi|^p\U^p\,d\mu+Ck^{j(p-1)}\omega_{f_1}\left(\frac{1}{2\tilde{C}k^{j(p-1)}}\right)\int_{B_r}\varphi^p\mathcal{U}^p\,d\nu.
\end{align*}
From the Sobolev inequality \eqref{S} we find,
\begin{align}
\label{3578}&\left( \frac{1}{\nu(B_r)}\int_{B_{r}}| \varphi \mathcal U|^{kp}d\nu\right)^{\frac{1}{k}}\\
\nonumber&\quad\le C \frac{r^p}{\mu(B_{r})}\left\{\int_{B_{r}}|\sqrt Q\nabla  \varphi|^p\U^p \,d\mu+k^{j(p-1)}\omega_{f_1}\left(\frac{1}{2\tilde{C}k^{j(p-1)}}\right)\int_{B_{r}}|\varphi \mathcal U|^pd\nu\right\}.
\end{align}
We specialize our choice of cutoff function $\varphi$ to be $\varphi_j$ as provided by condition (H6) relative to the ball $B_r$. From \eqref{3578} and (H6) we obtain with $S_j=\operatorname{supp}(\varphi_j)$, $S_0=B_r$,
\begin{align}
\label{gnegnegne}&\left(\frac{1}{\nu(B_{r})} \int_{B_{r}}|\mathcal U|^{kp}\chi_{S_{j+1}}d\nu\right)^{\frac{1}{k}}\\
&\le C r^{p(1-\gamma)}\frac{\nu(B_r)}{\mu(B_{r})}
\left\{T^{jp}+k^{j(p-1)}r^{\gamma p}\omega_{f_1}\left(\frac{1}{2\tilde{C}k^{j(p-1)}}\right)\right\}\!\!\!
\left(\frac{1}{\nu(B_r)}\int_{B_{r}}|\mathcal U|^p\chi_{S_j} d\nu\right).\nonumber
\end{align}
Now recall that $\Phi_{f_1}(r)=\frac{1}{2\sigma}\geq\frac{1}{2\tilde{C}k^{j(p-1)}}$, hence $$r^{-1}\Phi_{f_1}^{-1}\left(\frac{1}{2\tilde{C}k^{j(p-1)}}\right)\leq r^{-1}\Phi_{f_1}^{-1}\left(\frac{1}{2\sigma}\right)\leq1.$$
Thus
$$k^{j(p-1)}r^{\gamma p}\omega_{f_1}\left(\frac{1}{2\tilde{C}k^{j(p-1)}}\right)
=\frac{1}{2\tilde{C}}r^{\gamma p} \left(\Phi_{f_1}^{-1}\left(\frac{1}{2\tilde{C}k^{j(p-1)}}\right)\right)^{-\gamma p}\geq \frac{1}{2\tilde{C}}.$$
From \eqref{gnegnegne} then we have
\begin{align*}
&\left(\frac{1}{\nu(B_{r})} \int_{B_{r}}|\mathcal U|^{kp}\chi_{S_{j+1}}d\nu\right)^{\frac{1}{k}}\\
&\le C r^{p}\frac{\nu(B_r)}{\mu(B_{r})}T^{jp}\left(\Phi_{f_1}^{-1}\left(\frac{1}{2\tilde{C}k^{j(p-1)}}\right)\right)^{-\gamma p}
\left(\frac{1}{\nu(B_r)}\int_{B_{r}}|\mathcal U|^p\chi_{S_j} d\nu\right).
\end{align*}
Now for $j\in\mathbb{N}$, set $$\Phi(s,j)=\left(\frac{1}{\nu(B_r)}\int_{B_r}|\bar{u}|^s\chi_{S_j} d\nu\right)^{1/s}.$$
Thus, our previous inequality can be rewritten using $\mathcal U = \bar{u}^{q_j}$ and condition \eqref{3456} as
\begin{align}
\label{008}\Phi(kpq_j,j+1)^{pq_j} \leq C r^{\gamma p} T^{jp}\left(\Phi_{f_1}^{-1}\left(\frac{1}{2\tilde{C}k^{j(p-1)}}\right)\right)^{-\gamma p}\Phi(pq_j,j)^{pq_j},
\end{align}
that is
\begin{align*}
&\Phi(k^{j+1}p_0,j+1)\\
&\leq (CT^{pj}r^{\gamma p})^\frac{1}{p_0k^j}\left(\Phi_{f_1}^{-1}\left(\frac{1}{2\tilde{C}k^{j(p-1)}}\right)\right)^{-\frac{\gamma p}{p_0k^j}}\Phi(k^jp_0,j).
\end{align*}
We now iterate the inequality for $j\in\mathbb{N}$, using Lemma \ref{lemmaserie} and noting that $\Phi_{f_1}(r)=\frac{1}{2\tilde{C}}$, arguing as in the proof of \eqref{2323}. We obtain
\begin{align}
\label{3141}\sup_{B_{\tau r}}\bar{u}&\leq Cr^\frac{\gamma kp}{p_0(k-1)}\prod_{j=0}^\infty\left(\Phi_{f_1}^{-1}\left(\frac{1}{2\tilde{C}k^{j(p-1)}}\right)\right)^{-\frac{\gamma p}{p_0k^j}}\left(\frac{1}{\nu(B_r)}\int_{B_r}\bar{u}^{p_0}\,d\nu\right)^\frac{1}{p_0}\\
\nonumber &\leq C e^{\frac{\gamma kp}{p_0(k-1)}M_1}\left(\frac{1}{\nu(B_r)}\int_{B_r}\bar{u}^{p_0}\,d\nu\right)^\frac{1}{p_0},
\end{align}
with $M_1>0$ satisfying \eqref{12356}.

We proceed in a similar way in case $\beta<1-p$, choosing $\beta_j=-p_0k^j-p+1$ for any $j\in\mathbb{N}$ such that $q_j=\frac{\beta_j+p-1}{p}=-\frac{p_0}{p}k^j<0$, with $p_0$ as in \eqref{johnnirenberg}. Up to choosing an even larger constant $\tilde{C}\geq1$ we again obtain \eqref{008}, this time for negative $q_j$'s. Then we deduce
\begin{align*}
&\Phi(-k^{j+1}p_0,j+1)\geq\\
&(CT^{pj}r^{\gamma p})^{-\frac{1}{p_0k^j}}\left(\Phi_{f_1}^{-1}\left(\frac{1}{2\tilde{C}k^{j(p-1)}}\right)\right)^{\frac{\gamma p}{p_0k^j}}\Phi(-k^jp_0,j).
\end{align*}
Iterating the inequality for $j\in\mathbb{N}$ we obtain
\begin{align*}
\inf_{B_{\tau r}}\bar{u}&\geq Cr^{-\frac{\gamma kp}{p_0(k-1)}}\prod_{j=0}^\infty\left(\Phi_{f_1}^{-1}\left(\frac{1}{2\tilde{C}k^{j(p-1)}}\right)\right)^{\frac{\gamma p}{p_0k^j}}\left(\frac{1}{\nu(B_r)}\int_{B_r}\bar{u}^{-p_0}\,d\nu\right)^{-\frac{1}{p_0}}.
\end{align*}
Hence using Lemma \ref{lemmaserie} we have
\begin{align}
\label{0141}&C e^{\frac{\gamma kp}{p_0(k-1)}M_1}\inf_{B_{\tau r}}\bar{u}\geq \left(\frac{1}{\nu(B_r)}\int_{B_r}\bar{u}^{-p_0}\,d\nu\right)^{-\frac{1}{p_0}}.
\end{align}
Now using \eqref{johnnirenberg}, from \eqref{3141} and \eqref{0141} we deduce that
$$
\sup_{B_{\tau r}}\bar{u}\leq  C e^{\frac{2\gamma kp}{p_0(k-1)}M_1} \inf_{B_{\tau r}}\bar{u}.
$$
Therefore, recalling the definition of $\bar{u}$,
$$
\sup_{B_{\tau r}}u\leq C_2 e^{C_3M_1}  \left(\inf_{B_{\tau r}}u+\lambda\right).
$$
If $\Phi_{|f|^{p-1}}(r)>0$ the proof is complete. The proof in case $\Phi_{|f|^{p-1}}(r)=0$ follows along the same lines, and therefore we will omit it.
\end{proof}

Now, as a simple consequence of the Harnack inequality, we get some regularity results for weak solutions of \eqref{Qppoisson}. The proof is a standard consequence of the Harnack inequality, see e.g. \cite{s}, so we will omit it.

\begin{thm}\label{lastone}
Let $1<p<\infty$.  Assume condition (H1)-(H6). Suppose that $u\in QW^{1,p}(\Omega)$ is a solution of \eqref{Qppoisson} where $|f|^{p-1}\in M_p'(\Omega)$ and $f\in L^{p-\frac{p(k-1)}{pk-1}}_{\text{loc}}(\Omega,d\nu)$, with $k$ as in (H5).

Assume also that there exists $r^*,C>0$ such that
\begin{equation}\label{34567}
\sup r^{p(1-\gamma)}\frac{\nu(B_r)}{\mu(B_{r})}\leq C
\end{equation}
where the supremum is computed over all balls such that ${8}B\Subset\Omega$ and with radius $r\leq r^*$, with $\gamma$ as in (H6).

Suppose moreover that there exists $M_1>0$ such that condition
$$
\int_0^{r}\frac{\left(\Phi_{|f|^{p-1}}(s)\right)^{\frac{1}{p-1}}}{s}\,ds\leq M_1 \left(\Phi_{|f|^{p-1}}(r)\right)^{\frac{1}{p-1}}.
$$
is satisfied for every small enough $r>0$ and that
$$
(\Phi_{|f|^{p-1}}(r))^{\frac{1}{p-1}}\leq Cr^\alpha
$$
for some positive constants $C,\alpha$. Then $u$ is locally H\"older continuous with respect to the metric $\rho$.
\end{thm}

\begin{rem}\label{lastrem}
We explicitly note that, as a consequence of Theorem \ref{lastone}, $u$ is also continuous with respect to the Euclidean topology. Moreover, if for every subdomain $\Omega'\Subset\Omega$ there are positive constants $c_0,\delta$ such that for every $x,x_0\in\Omega'$ one has
$$
d(x,x_0)\leq c_0|x-x_0|^\delta,
$$
then $u$ is locally H\"older continuous also with respect to the Euclidean distance.
\end{rem}

\bigskip
{\bf Acknowledgement}.
This work has been supported by Università degli Studi di Catania, "Piano PIA.CE.RI", upb 53722122154.

G. Di Fazio, M.S. Fanciullo, D.D. Monticelli and P. Zamboni are also members of the Gruppo Nazionale per l’Analisi Matematica, la Probabilità e le loro Applicazioni
(GNAMPA) of the Istituto Nazionale di Alta Matematica (INdAM).

 \bibliographystyle{plain}

\end{document}